\documentclass[12pt,a4paper]{amsart}
\usepackage[utf8]{inputenc}
\usepackage[T1]{fontenc}
\usepackage{lmodern}
\usepackage[english]{babel}
\usepackage{amsmath,mathtools,amsthm,amssymb,amsfonts}
\usepackage{color}
\usepackage{geometry}
\usepackage{mathrsfs}
\usepackage{graphicx}
\usepackage{enumitem}
\usepackage{tikz} 
\usepackage{pgfplots}
\usepackage{esvect}
\usepackage{pdfpages}
\usetikzlibrary{shapes,arrows}
\usetikzlibrary{calc}
\usetikzlibrary{shapes.geometric}
\usetikzlibrary{shapes.arrows}
\usetikzlibrary{arrows.meta}
\usetikzlibrary{decorations.markings}
\usetikzlibrary{fit}
\usetikzlibrary{patterns}
\usetikzlibrary{hobby}
\usepgfplotslibrary{patchplots}
\pgfplotsset{compat=1.11}
\usepackage{hyperref}
\usepackage{nicefrac}
\usepackage{cancel}
\usepackage{parskip} % no indent (sangria)

\theoremstyle{plain}
\newtheorem{theorem}{Theorem}[section]
\newtheorem{corollary}[theorem]{Corollary}

\newtheorem{proposition}[theorem]{Proposition}
\newtheorem{lemma}[theorem]{Lemma}
\theoremstyle{definition}
\newtheorem{definition}[theorem]{Definition}
\newtheorem{remark}[theorem]{Remark}

\newtheorem*{theorem*}{Theorem}
\newtheorem*{definition*}{Definition}
\newtheorem*{corollary*}{Corollary}
\newtheorem*{remark*}{Remark}
\newtheorem*{thm*}{Theorem}
\newtheorem*{conjecture*}{Conjecture}

\newcommand{\definedas}{\mathrel{\raise.095ex\hbox{\rm :}\mkern-5.2mu=}}
\newcommand{\asdefined}{\mathrel{=\mkern-5.2mu\raise.095ex\hbox{\rm :}}}

\newcommand{\R}{\mathbb{R}}

\newcommand{\diver}{\operatorname{div}}
\newcommand{\tr}{\operatorname{tr}}

\newcounter{flabelcounter}
\allowdisplaybreaks

\title[]{Foliations by Critical Surfaces of the Hawking Energy in Asymptotically Flat Initial Data Sets}
\author[Pe\~nuela Diaz]{Alejandro Pe\~nuela Diaz}
\address{ University of Potsdam, 14476 Potsdam, Germany}
\email{alejandro.penuela.diaz@uni-potsdam.de}

\begin{document}
\begin{abstract}
Area-constrained critical surfaces for the Hawking quasi-local energy (“Hawking surfaces”) provide a natural setting for that energy: they enjoy positivity and rigidity properties, as shown in earlier work \cite{rigidiaz}. We construct large-scale foliations at infinity by Hawking surfaces in asymptotically Schwarzschild initial data sets. Using a Lyapunov–Schmidt reduction within the Willmore foliation framework of Eichmair–Koerber \cite{eichko}, we prove existence and uniqueness of the foliation and study its coordinate center. Under the dominant energy condition, we show that along the leaves of the foliation, the Hawking energy is positive and converges to the ADM energy in the large-sphere limit; moreover, subject to an explicit integral constraint, it is monotone along the foliation. Under weaker assumptions we construct an on–center family of Hawking surfaces that, while not necessarily a foliation, still enjoys positivity and the large–sphere limit. Finally, we obtain a rigidity statement and verify that our hypotheses hold in a broad class of data, initial data sets with harmonic or York asymptotics, thereby demonstrating the robustness of Hawking surfaces as a quasi-local energy tool in dynamical spacetimes. 
\end{abstract}
\maketitle
\section{Introduction and Main Results}

A fundamental problem in classical general relativity is the search for a good definition of quasi-local energy. While for an isolated system (asymptotically flat), we can measure the total energy of that system through the ADM energy, there is no canonical way to assign an energy or mass content to an arbitrary bounded region. For example, given a closed 2-sphere $\Sigma$ sitting inside an initial data slice (spacelike manifold), one cannot directly read off how much mass or energy lies in the enclosed volume. Mathematically, this problem becomes one of finding a functional, built only from the intrinsic and extrinsic geometry of $\Sigma$, that exhibits the key physical properties one expects of energy.

Over the years many definitions of quasi-local energy have been proposed (see \cite{Living} for a comprehensive review), each with distinct strengths and limitations. To be physically acceptable, these definitions must satisfy certain physical conditions, particularly under the dominant energy condition, a realistic model for matter. These conditions include:
\begin{enumerate}[label=(\arabic*), ref=(\arabic*)]
    \item\label{item:nonneg}  Nonnegativity: The energy measure must always be nonnegative.

     \item\label{item:rigidity} Rigidity: The energy measure should vanish if and only if the enclosed region is flat. This ensures that quasi-local energy distinguishes between flat and curved spacetimes. 

     \item\label{item:monotone} Monotonicity: The quasi-local energy should increase monotonically when evaluated on a sequence of expanding spheres.

    \item\label{item:large} Large-sphere limit:  As the quasi-local energy is evaluated on increasingly large spheres, it should asymptotically approach the ADM energy, representing the total energy of the system. 

    \item\label{item:small} Small-sphere limit:  On spheres shrinking to a point along a light cone, the quasi‐local energy (after rescaling) should converge to the Einstein tensor evaluated on the  timelike direction defining the cut. This reflects the fact that it captures the local energy density.
\end{enumerate}

These requirements translate into precise mathematical conditions that any viable quasi-local energy functional must satisfy, ensuring both rigorous consistency and genuine physical relevance.

In this paper we work in the initial data set setting, this means that we consider a smooth $3$-dimensional Riemannian manifold $(M,g)$, equipped with a symmetric $2$-tensor $k$, we refer to the triple $(M,g,k)$ as an initial data set. In the context of general relativity, an initial data set represents a spacelike hypersurface with second fundamental form $k$ in a 4-dimensional spacetime (Lorentzian manifold).

Among quasi-local candidates, a prominent one is the one described by Hawking in \cite{Hawma}, the so-called \emph{Hawking energy},  which can be written for a surface $\Sigma \subset M $ as
\begin{equation}\label{hawkingmass2.intro}
    \mathcal{E}(\Sigma) = \sqrt{\frac{|\Sigma|}{16 \pi}} \left( 1- \frac{1}{16 \pi} \int_\Sigma H^2 - P^2 d\mu \right),
\end{equation}
where $H$ is the mean curvature of the surface $\Sigma$ and   $P=\tr_{g_{\Sigma}} k$ is the trace of the tensor $k$ with respect to the metric induced on $\Sigma$, that is  $P= \tr_\Sigma k= \tr k -k(\nu,\nu)$, where $\nu$ is the outward  normal to $\Sigma$ in $M$.

The Hawking energy is a natural and relatively simple quasi-local energy functional: it captures correct local behavior (e.g., the small‐sphere limit, as shown by Horowitz and Schmidt \cite{Haw},  item  \ref{item:small}   of the list above) but fails to be nonnegative on arbitrary surfaces. For instance, in Euclidean space  every non-round 2-sphere has strictly negative Hawking energy, so one must single out a distinguished family of surfaces if one wants positivity.

The first landmark result in this direction is due to Christodoulou–Yau \cite{Chriyau}, who proved that, in the time‐symmetric ($k=0$) setting under the dominant energy condition, the Hawking energy is nonnegative on stable constant mean curvature spheres. This suggests seeking a canonical class of surfaces, beyond the time-symmetric case, on which positivity (and ultimately rigidity, monotonicity, and the large-sphere limit) can be restored.

A natural approach is to adopt a variational framework: for fixed area,  studying (\ref{hawkingmass2.intro}) is equivalent to studying the Hawking functional 
\begin{equation}\label{hawfun}
    \mathcal{H}(\Sigma)= \frac{1}{4} \int_\Sigma H^2 - P^2 d\mu
\end{equation}
We consider area-constrained critical surfaces of this functional. In case $k=0$, in a totally geodesic hypersurface, the Hawking functional reduces to the Willmore functional 
\begin{equation*}
    \mathcal{W}(\Sigma)= \frac{1}{4} \int_\Sigma H^2  d\mu 
\end{equation*}
and the critical surfaces of this functional subject to the constraint that $|\Sigma| $ be fixed are the area-constrained Willmore surfaces which we call here for simplicity just \emph{Willmore surfaces}. These surfaces are characterized by the following Euler-Lagrange equation with the Lagrange parameter $\lambda$.
\begin{equation}\label{Willeq}
    0= \lambda H  +\Delta^\Sigma H + H|\mathring{B}|^2+ H \mathrm{Ric}(\nu, \nu), 
\end{equation}
where $\mathring{B}$   is the traceless part of the second
fundamental form $B$ of $\Sigma$ in $M$, that is $ \mathring{B} = B- \frac{1}{2} H g_\Sigma$ with norm  $|\mathring{B}|^2 = \mathring{B}_{ij}\, g_\Sigma^{ip}\, g_\Sigma^{jq}\, \mathring{B}_{pq}$,  $\mathrm{Ric}$ is the Ricci curvature of $M$, $\nu$ is the outward normal to $\Sigma$  and $\Delta^\Sigma $ is the Laplace-Beltrami operator on $\Sigma$.

Area-constrained Willmore surfaces have been extensively studied, and in the context of general relativity, a key development was made in \cite{willflat} by Lamm, Metzger, and Schulze. They proved that an asymptotically Schwarzschild manifold admits a unique foliation near infinity by Willmore spheres—a “foliation at infinity’’ covering the end outside a compact set. In their framework, these surfaces are singled out as optimal for evaluating the Hawking energy since under nonnegative scalar curvature (dominant energy condition for $k=0$), the Hawking energy is nonnegative on Willmore surfaces, and it is monotonically nondecreasing along the foliation toward infinity.

Koerber \cite{Thomas} further showed that the leaves of the Willmore foliation are strict local area-preserving maximizers of the Hawking energy. Later, Eichmair and Koerber refined the construction of the foliation by Willmore spheres at infinity in~\cite{eichko}, employing a different approach based on a Lyapunov--Schmidt reduction procedure, a method that we will also use in this paper. Moreover, in~\cite{willcen}, they investigated the notion of center of mass naturally associated with this foliation.

In \cite{willflat} Lamm--Metzger--Schulze construct the foliation  through a direct perturbative analysis of large coordinate spheres in the Schwarzschild metric: the authors explicitly compute the first and second variations of the Willmore functional, derive a priori curvature and position estimates, and establish the invertibility of the linearized operator via comparison with the Schwarzschild case, allowing the use of the implicit function theorem to obtain existence and uniqueness under small perturbations of the metric. In contrast, Eichmair--Koerber \cite{eichko} use a more abstract functional-analytic approach based on Lyapunov--Schmidt reduction, treating the area-constrained Willmore equation as a finite-dimensional critical point problem. By perturbing large spheres $S_\lambda(\lambda \xi)$ to nearby approximate solutions and identifying true solutions as critical points of a reduced functional $G_\lambda(\xi)$, they avoid the need for smallness assumptions, requiring only mild asymptotic conditions on the scalar curvature. This method yields broader existence and optimal uniqueness results, extending the foliation theory to a more general class of asymptotically Schwarzschild manifolds.

In the dynamical setting ($k\neq 0$), the natural generalization of area-constrained Willmore surfaces is the area-constrained critical surfaces of the Hawking functional (\ref{hawfun}).   
\begin{definition}
We call area-constrained critical surfaces of the Hawking functional $ \int_\Sigma H^2 -P^2 d\mu$, \emph{Hawking surfaces}. These surfaces are characterized by the equation 
 \begin{equation}\label{eulag}
\begin{split}
   0=& \lambda H  +\Delta^\Sigma H + H|\mathring{B}|^2+ H \mathrm{Ric}(\nu, \nu)+P( \nabla_\nu \tr k - \nabla_\nu k(\nu,\nu )) - 2P \diver_\Sigma (k(\cdot, \nu))\\ & +\frac{1}{2}H P^2 - 2k (\nabla^\Sigma P, \nu )
    \end{split}
\end{equation}
for some real parameter $\lambda$.
\end{definition}
Previously studied in \cite{Alex, diaz2023local, rigidiaz}, these surfaces are well adapted to the Hawking energy: they have positivity and rigidity properties under the dominant energy condition \cite{rigidiaz} (items \ref{item:nonneg} and \ref{item:rigidity} of list above).

In the dynamical setting, local foliations by Hawking surfaces near a point were constructed in~\cite{diaz2023local}, where the small-sphere limit was also analyzed. At the other extreme, Friedrich’s Ph.D. thesis \cite{Friedrich2020}, building on the Willmore foliation framework of Lamm, Metzger, and Schulze \cite{willflat}, established large-scale foliations near infinity using the same perturbative and implicit-function-theorem method developed in their work.

Beyond monotonicity and the large-sphere limit, an important complementary question for foliations at infinity is whether such a foliation also recovers a natural center of mass. In the time–symmetric case, Huisken–Yau showed that the CMC foliation yields the Hamiltonian center of mass \cite{HY}, and Eichmair–Koerber under more restrictive assumptions proved the analogous statement for the Willmore foliation \cite{willcen}. This prompts the question: under what additional symmetry or decay conditions does a foliation by Hawking surfaces become asymptotically centered, and does its coordinate center coincide, in the dynamical setting, with the spacetime  (STCMC) center of mass introduced by Cederbaum and Sakovich \cite{STCMC}?

This paper has two goals: to extend the Eichmair–Koerber foliation \cite{eichko} to dynamical data, and to establish the main physical properties of the Hawking energy along the resulting surfaces. We construct global foliations by Hawking surfaces in asymptotically Schwarzschild ends and analyze their principal features: existence, uniqueness, and asymptotic center. Along these foliations, the Hawking energy is positive, satisfies the large–sphere limit, admits a rigidity statement, and under further technical hypotheses is monotone. Under weaker assumptions, we also construct an on–center family of Hawking surfaces, which is not a foliation in general, but still satisfies positivity and the large–sphere limit. Taken together with the positivity and rigidity theory of \cite{rigidiaz} and the small–sphere limit of \cite{Haw}, our results show that, subject to technical assumptions, the Hawking energy satisfies the five main physical criteria for a quasi-local energy: four on Hawking surfaces and the small-sphere limit along light-cone cuts.
\subsection{Main results}

We restrict our analysis to initial data sets $(M,g,k)$ and construct the surfaces in the asymptotically Schwarzschild setting. Note that these  manifolds  can be physically interpreted as models for isolated astrophysical
systems.

\textbf{Notation and asymptotic setting.} We  denote by $\delta$ and $\rho$ the Euclidean metric and  radial direction $\frac{x}{|x|}$ respectively. For functions \(f,g\), we write
\[
f(x)=\mathcal{O}(g(x))\quad(x\to\infty)
\]
if there are constants \(C,\hat{\delta}>0\) such that \(|f(x)|\le C\,|g(x)|\) for all \(x\ge\hat{\delta}\), and
\[
f(x)=o(g(x))\quad(x\to\infty)
\]
if \(\lim_{x\to\infty}\frac{f(x)}{g(x)}=0\)
(equivalently: for every \(\varepsilon>0\) there exists \(\hat{\delta}>0\) with \(|f(x)|\le\varepsilon\,|g(x)|\) for all \(x\ge\hat{\delta}\)).
 Also, \(f(x)\le \mathcal{O}(g(x))\) means that there exist \(C,\hat{\delta}>0\) such that \(f(x)\le C\,|g(x)|\) for all \(x\ge\hat{\delta}\).

We denote the even and odd parts of $f$ by
\[
f^{\mathrm{odd}}(x) := \tfrac{1}{2}(f(x) - f(-x)), \qquad
f^{\mathrm{even}}(x) := \tfrac{1}{2}(f(x) + f(-x)).
\]
and for tensors we take even/odd parts \emph{componentwise} in the chosen chart.  In what follows, all estimates are expressed in these coordinates, and all even/odd decompositions and $\mathcal{O}$–bounds are understood with respect to the asymptotically flat chart from the following definition.

\begin{definition}\label{asymptoticsc}
Let \((M,g,k)\) be a complete initial data set for the Einstein equations.  We say \((M,g,k)\) is \(\mathcal C^4\)-\emph{asymptotically Schwarzschild} of mass \(m>0\) if there exists a compact set \(K\subset M\) and a diffeomorphism
\[
\Phi: M\setminus K \;\longrightarrow\; \{\,x\in\R^3: |x|>R\}
\]
for some \(R>0\), such that in the pulled–back Cartesian coordinates \(x=(x^1,x^2,x^3)\) on \(\{\,|x|>R\}\) the metric and second fundamental form satisfy
\[
g_{ij}(x)
=\Bigl(1+\tfrac{m}{2|x|}\Bigr)^{4}\,\delta_{ij}
\;+\;\sigma_{ij}(x),
\qquad
\partial^J\sigma_{ij}(x)
=\mathcal{O}\bigl(|x|^{-2-|J|}\bigr)
\quad(|J|\le4),
\]
 and
\[
k_{ij}(x)
=\mathcal{O}\bigl(|x|^{-2}\bigr),
\qquad
\partial^J k_{ij}(x)
=\mathcal{O}\bigl(|x|^{-2-|J|}\bigr)
\quad(|J|\le2).
\]
where $\delta$ denotes the Euclidean metric and $\sigma$ is a $2$-tensor. We call \(K\) (or its image under \(\Phi\)) the \emph{interior region} of \(M\).\end{definition}
The foliations that we are considering are composed of on-center spheres.
\begin{definition}
   In an asymptotically Schwarzschild or  flat manifold, we  say that a surface $\Sigma \subset M $ is \emph{on-center} if it bounds a
compact region that contains the interior region of the manifold. 
\end{definition}
\noindent\textbf{Theorem \ref{exisfoli}.}
   \emph{Let $(M,g,k)$ be $ \mathcal{C}^4$-asymptotic to Schwarzschild with mass $m>0$. There exist constants $ \Tilde{\eta}_i>0$, i=1,2,3 such that if for $\eta_i< \Tilde{\eta}_i $ it holds 
\begin{equation}
    x^i \partial_i (|x|^2 \mathrm{Sc}) \leq \mathcal{O} (\eta_1 |x|^{-2}),\quad  \mathrm{Sc}^{\text{odd}} = \mathcal{O} (\eta_2 |x|^{-4})
\end{equation}
and 
\begin{equation}
k_{ij}(x) = \mathcal{O}(\eta_3 |x|^{-2}) , \quad (\nabla k)_{ij}(x)= \mathcal{O}(\eta_3 |x|^{-3}).
\end{equation}
 Then there exists a compact set \(K \subset M\), a number \(\lambda_0 > 0\),  and on-center stable area-constrained critical
spheres \(\Sigma(\lambda)\) of the Hawking functional  with parameter \(\lambda \in (0, \lambda_0)\), such that \(M \setminus K\) is foliated by the family
\(\{\Sigma(\lambda) : \lambda \in (0, \lambda_0)\}\).}

 We also obtain a uniqueness result for the previous foliation similar to the one obtained in \cite{eichko}.
 
\noindent\textbf{Theorem \ref{uniqueness}.} \emph{Under the same assumptions as the previous theorem. There exists a small constant \(\epsilon_0 > 0\) and a compact set
\(K \subset M\) which depend only on \((M, g, k)\) such that the following holds. For every \(\Tilde{\delta} > 0\), there exists a
large constant \(r_0 > 1\) such that every area-constrained critical sphere \(\Sigma \subset M \setminus K\) of the Hawking functional with
\[
|\Sigma| > 4 \pi r_0^2, \quad \quad \Tilde{\delta} r(\Sigma) < R(\Sigma), \quad \Tilde{\delta} R(\Sigma) < r(\Sigma)  \quad \text{and}\quad \int_{\Sigma} |\mathring{B}|^2 \, d\mu < \epsilon_0 
\]
 where $R(\Sigma):=\sup_{x \in \Sigma } |x|$ and $r(\Sigma)$ is the area radius, belongs to the foliation from Theorem \ref{exisfoli}.}

We can also relax our assumptions on $\mathrm{Sc}$ and $k$, obtaining a family of surfaces that is not necessarily a foliation, i.e., the surfaces may intersect with each other. 

 \noindent\textbf{Theorem \ref{family}.} \emph{Let $(M,g,k)$ be $ \mathcal{C}^4$-asymptotic to Schwarzschild with mass $m>0$. There exist constants $ \Tilde{\eta}_i>0$, i=1,2 such that if for $\eta_i< \Tilde{\eta}_i $ it holds 
     \begin{equation}
         \mathrm{Sc} \geq -\mathcal{O} (\eta_1 |x|^{-4}).
     \end{equation}
  Then there exists a compact set \(K \subset M\), a number \(\lambda_0 > 0\), and a family
\(\{\Sigma(\lambda) : \lambda \in (0, \lambda_0)\}\) of on-center stable area-constrained critical
spheres \(\Sigma(\lambda)\) of the Hawking functional,  satisfying (\ref{eulag}) with parameter \(\lambda\) such that as $\lambda \to 0$, $|\Sigma(\lambda)|\to \infty$ and   $ R(\lambda) \to \infty$, where $R(\lambda):=\sup_{x \in \Sigma(\lambda) } |x|$.} 

\emph{If additionally it holds,
\begin{equation}
     x^i \partial_i (|x|^2 \mathrm{Sc}) \leq \mathcal{O} (\eta_1 |x|^{-2}),
\end{equation}
\begin{equation}
 k(\cdot, \rho) = \mathcal{O} ( \eta_2| x|^{-2}), \quad \pi (\rho, \rho) = \mathcal{O} (\eta_2  | x|^{-2})\quad \text{and} \quad (\nabla_{\rho} \pi) (\rho, \rho)= \mathcal{O} (\eta_2| x|^{-3}),
\end{equation}
where $\pi = k - \tr k\, g$ and $\rho= \frac{x^i}{|x|}\partial_i$. Then the family of surfaces is unique in the sense of Theorem \ref{uniqueness}.}

In Theorems \ref{largesphere}, \ref{positivity} we prove that the  Hawking energy is nonnegative and satisfies the large–sphere limit ($ \mathcal{E}(\Sigma_r)\to m $ as $r \to \infty$) for both the foliation of Theorem~\ref{exisfoli} and the on–center family of Theorem~\ref{family}.

 In a similar way to \cite{willcen} for Willmore foliations, we also study the center of the foliations constructed above. Because the foliation is composed of on-center spheres, any natural geometric center lies in the compact interior region \(K\) from Definition \ref{asymptoticsc}, which is not visible through the asymptotic coordinate chart. For this reason, we introduce instead the \emph{coordinate center} of the foliation, defined at infinity via the ambient chart.

A closed, oriented \(2\)-surface \(\Sigma \hookrightarrow \mathbb{R}^3\) has a Euclidean coordinate center \(\vec{c}(\Sigma)\) defined by
\begin{equation}
    \vec{c}(\Sigma) := \frac{1}{|\Sigma|_\delta} \int_{\Sigma} \vec{x} \, d\mu_\delta, \label{eq:euclidean_center}
\end{equation}
where \(d\mu_\delta\) and \(|\Sigma|_\delta\) are the Euclidean area element and area, respectively. Given an asymptotically flat coordinate chart \(\vec{x}: M \setminus B \to \mathbb{R}^3 \setminus B_R(0)\), this definition extends to closed, oriented \(2\)-surfaces \(\Sigma \subset M \setminus B\) by pushing \(\Sigma\) forward to \(\mathbb{R}^3\) and setting \(\vec{c}(\Sigma):=\vec{c}(\vec{x}(\Sigma))\). We call \(\vec{c}(\Sigma)\) the Euclidean center of \(\Sigma\) with respect to \(\vec{x}\).
\begin{definition}[\textbf{Coordinate center of a foliation}]
    Let \((M, g, k)\) be an asymptotically flat initial data set, and fix an asymptotically flat chart \(\vec{x}: M \setminus B \to \mathbb{R}^3 \setminus B_R(0)\). Let \(\{\Sigma_\sigma\}_{\sigma>\sigma_0}\) be a foliation of the end \(M \setminus B\) by closed, oriented surfaces with area radius \(r(\Sigma_\sigma):=\sqrt{|\Sigma_\sigma|/4\pi} \to \infty\) as \(\sigma \to \infty\). Denote by \(\vec{c}(\Sigma_\sigma)\) the Euclidean center of \(\Sigma_\sigma\) with respect to \(\vec{x}\). Then the coordinate center \(\vec{C}=(C^1,C^2,C^3)\) of the foliation (with respect to \(\vec{x}\)) is
    \[
        \vec{C} := \lim_{\sigma\to\infty} \vec{c}(\Sigma_\sigma),
    \]
    provided the limit exists. Otherwise, we say the coordinate center of the foliation diverges (with respect to \(\vec{x}\)).
\end{definition}

\noindent\textbf{Theorem \ref{centerfolia}.} \emph{Let $(M,g,k)$ be $ \mathcal{C}^4$-asymptotic to Schwarzschild with mass $m>0$, satisfying the conditions of the foliation  Theorem \ref{exisfoli}.  Furthermore, assume that 
   \begin{equation*}
   \mathrm{Sc}^{\text{odd}} =\mathcal{O} ( |x|^{-5}), \quad  |k^{\text{even}}| + |x||(\nabla k)^{\text{odd}}| = \mathcal{O} (|x|^{-3})
\end{equation*}
and 
\begin{equation*}
    x^i \partial_i (|x|^2 \mathrm{Sc}) \leq \mathcal{O}( |x|^{-3}),
\end{equation*}
Then there exists an on-center foliation by Hawking surfaces, and its center is given by 
\begin{equation*}\begin{split}
    C_{f}=&  \lim_{r \to \infty}\frac{1}{128 \pi} r^3 \int_{S_r(r \xi(r))} \left( \mathrm{Sc} +r \Big( P( \nabla_\nu \tr k - \nabla_\nu k(\nu,\nu )) - 2 \diver_\Sigma (P k(\cdot, \nu))  +\frac{1}{2}H P^2 \Big)\right) \nu \, d\mu_{\delta}\\
    &+ C_H
    \end{split}
\end{equation*}
provided the limit converges, where $C_H$ is the Hamiltonian center of mass, and the spheres $S_r(r \xi(r))$ are introduced in the proof of the previous foliation construction theorem.} 

Moreover, in Corollaries \ref{corocenter1} and \ref{corocenter2} we identify special cases where the foliation’s center converges more directly, and can be computed without reference to the intermediate spheres $S_r(r \xi(r))$ used in the construction. In particular, we find that the center of the  foliation is more sensitive to the asymptotic distribution of the scalar curvature \(\mathrm{Sc}\) and of the second fundamental form \(k\) than the STCMC center of Cederbaum–Sakovich, and it need not coincide with the latter in general.

In Section \ref{rigifoligene} we study the nonnegativity and monotonicity of the Hawking energy along the foliation and prove the following  result.

\noindent\textbf{Theorem \ref{monotinic0}.} \emph{Let $(M,g,k)$ be an initial data set  satisfying the dominant energy condition and let \( \Sigma \) be a Hawking surface with positive mean curvature satisfying  $\int_\Sigma f \, d\mu \leq 0$ for
    \begin{equation*}
        f:= \left( \frac{P}{H}\right)^2|k|^2+ \frac{1}{2 }(\tr k)^2    - \frac{3}{4} P^2- \frac{P}{H}( \nabla_\nu \tr k - \nabla_\nu k(\nu,\nu )) -  \frac{1}{2} |k|^2  -\frac{1}{2} |\mathring{B}|^2 -|J|.
    \end{equation*}
     Then if \( F : \Sigma \times [0, \varepsilon) \rightarrow M \) is a variation with initial velocity \( \frac{\partial F}{\partial s} \Big|_{s=0} = \alpha \nu \) and $\int_{\Sigma} \alpha H \, d\mu \geq 0,$ then 
\[
\frac{d}{ds} \mathcal{E}(F(\Sigma, s)) \geq 0.
\]}

In the harmonic and York asymptotic regimes, we verify that the assumptions on $k$ from our existence/monotonicity theory hold; in particular, the integrability condition $\int_\Sigma f\,d\mu \le 0$ is satisfied along the foliation. Consequently, in Theorems~\ref{monocosho}–\ref{monoyork} we show that the Hawking energy is nondecreasing along the foliation by Hawking surfaces for any manifold with these asymptotics.

Finally,  we obtain a rigidity result for   foliations by Hawking surfaces.

\noindent\textbf{Theorem \ref{rigigeneralfoli}.} \emph{Let $(M,g,k)$ be an asymptotically flat  initial data set satisfying the dominant energy condition. If the initial data set possesses an on-center foliation by Hawking surfaces   and one  surface $\Sigma$ of the foliation satisfies:}

    \begin{enumerate}[label=(\roman*)]
    
    \item \emph{The Hawking energy of the surface is zero.}
    
    \item \emph{Outside of $\Sigma$  $(M,g)$ has nonnegative scalar curvature.} 

    \item \emph{There exists a constant $0\leq\beta <\frac{1}{2}$ such that $\int_\Sigma f_\beta -\lambda \, d\mu \leq 0$  for 
 \begin{equation*}
       f_\beta:= \left( \frac{P}{H}\right)^2|k|^2+ \frac{1}{2 }(\tr k)^2   - \frac{3}{4} P^2- \frac{P}{H}( \nabla_\nu \tr k - \nabla_\nu k(\nu,\nu ))  -  \beta( |k|^2 + |\mathring{B}|^2 +2|J|).
    \end{equation*} }
\end{enumerate}    
\emph{Then $(M,g,k)$ is isometric to a spacelike hypersurface in  Minkowski spacetime.}

Note that here we only assume asymptotic flatness, a weaker hypothesis than asymptotic Schwarzschild, see Remark \ref{remarkrigidi}. The hypothesis of this Theorem hold for harmonic and York asymptotics, see  Corollary \ref{rigiharmoyor}.

In conclusion, our analysis shows that, under the dominant energy condition, the Hawking energy on Hawking surfaces has the expected quasi-local features in a broad class of initial data sets. Monotonicity, however,  currently holds only for foliations under an added integrability condition satisfied, for instance, in harmonic and York asymptotics. Moreover, as noted in \cite{rigidiaz}, the rigidity results may rely on unnecessarily strong hypotheses, inducing an “excess positivity”. Relaxing these constraints could yield a sharper characterization.

\vspace{0.305 cm}
Some of the material in Sections 2–4 first appeared in the author’s Ph.D. thesis \cite{PenuelaThesis}.
\vspace{0.4 cm}

 \section{Foliation construction and large-sphere limit }\label{large foliations}

\subsection{Foliation construction}

We  assume that our manifold is asymptotically 
Schwarzschild, and by scaling, we  assume $m=2$. Additionally, we  assume during this section that
\begin{equation}
    x^i \partial_i (|x|^2 \mathrm{Sc} ) \leq \mathcal{O}(\eta_1 |x|^{-2}), \quad  \mathrm{Sc}(x)-\mathrm{Sc}(-x)= \mathcal{O}(\eta_2 |x|^{-4})
\end{equation}
for some small constant $ \eta_1$, $\eta_2$ to be fixed later. Also assume that  
\begin{equation}\label{assumpi}
k_{ij}(x) = \mathcal{O}(\eta_3 |x|^{-2}) , \quad (\nabla k)_{ij}(x)= \mathcal{O}(\eta_3 |x|^{-3}) 
\end{equation}
for some small constants $ \eta_3$ to be fixed later. Note that  $x^i \partial_i (|x|^2 \mathrm{Sc} ) \leq \mathcal{O}(\eta_1 |x|^{-2})$ is a condition on the radial derivative of $|x|^2 \mathrm{Sc}$ and is equivalent to $\rho (|x|^2 \mathrm{Sc} ) \leq \mathcal{O}(\eta_1 |x|^{-3})$, where $\rho=\frac{x^i}{|x|}\partial_i $ is the radial vector in Euclidean space.

We rewrite the Euler–Lagrange equation \eqref{eulag} as 
\begin{equation}\label{shorteug}
    H\lambda+ W_1 + W_2=0.
\end{equation}
where
\begin{equation}\label{explaineulag}
   W_1 := \Delta^\Sigma H + H|\mathring{B}|^2 + H\,\mathrm{Ric}(\nu,\nu),
\quad
W_2 := P(\nabla_\nu \tr k - \nabla_\nu k(\nu,\nu))
      - 2\,\operatorname{div}_\Sigma\!\big(P\,k(\cdot,\nu)\big)
      + \tfrac12 H P^2, 
\end{equation}
and $W_1$ is the (standard) Willmore operator.

For a normal variation $ \frac{d F}{d t} = \alpha \nu $ one has the first-variation identity
\begin{equation}\label{firstvary}
\begin{split}
    \frac{d }{d t}\left( \int_\Sigma P^2 \,d\mu \right)&= \int_\Sigma 2 P \alpha ( \nabla_{\nu} \tr k - \nabla_{\nu} k(\nu,\nu )) +4 P    k(\nabla^\Sigma \alpha, \nu )
   +\alpha H P^2  d\mu\\
   &= \int_\Sigma   2 W_2 \alpha  d\mu
   \end{split}
\end{equation}
where we have used that $ \frac{d }{d t} P= \alpha ( \nabla_{\nu} \tr k - \nabla_{\nu} k(\nu,\nu )) +2     k(\nabla^\Sigma \alpha, \nu ) $. If we now take a second normal variation  $ \frac{d F}{d s} = \Tilde{\alpha} \nu$ and
vary the right hand side of (\ref{firstvary}), we obtain, after grouping some terms, the mixed second variation
\begin{equation}\label{secondvary}
    \begin{split}
   \frac{d }{d s}\frac{d }{d t}\left( \int_\Sigma P^2 \,d\mu \right)=\int_\Sigma &P^2 \alpha L_J \Tilde{\alpha} + 2  \left[ 2     k(\nabla^\Sigma \alpha,  \nu )- \alpha \nabla_\nu\pi(\nu,\nu) \right] \left[ 2    k(\nabla^\Sigma \Tilde{\alpha}, \nu ) -\Tilde{\alpha} \nabla_\nu\pi(\nu,\nu)\right] \\
&+ 2 P \alpha \left[   \nabla_{\nabla^\Sigma \Tilde{\alpha}}\pi(\nu,\nu) - \Tilde{\alpha} (\nabla_\nu\nabla_\nu\pi)(\nu,\nu)  +2 (\nabla_{\nu} k)(\nabla^\Sigma \Tilde{\alpha},\nu ) \right] \\    
   & +\Tilde{\alpha} H P \left[  4     k(\nabla^\Sigma \alpha, \nu )- 2  \alpha \nabla_\nu\pi(\nu,\nu)
   +\alpha H P  \right] + 4 P \Tilde{\alpha} \nabla_{\nu  }k(\nabla^\Sigma \alpha, \nu)  \\
   &+ 2 P \frac{\partial \alpha }{\partial s}  ( \frac{P}{2} H- \nabla_\nu\pi(\nu,\nu)) +4 Pk (\frac{\partial}{\partial s} \nabla^\Sigma \alpha, \nu) -4 P k(\nabla^\Sigma \alpha, \nabla^\Sigma \Tilde{\alpha} )\, d\mu\\
    \end{split}
\end{equation}
where $L_{J} \Tilde{\alpha}:= \frac{d }{d s}H=  - \Delta^\Sigma \Tilde{\alpha} - \left(|B|^2 + \mathrm{Ric} (\nu, \nu )\right)  \Tilde{\alpha} $ is the Jacobi operator, and  we set $\nabla\pi(\nu,\nu)=(\nabla k)(\nu,\nu ) -\nabla \tr k $ and $\nabla_\nu\nabla_\nu\pi(\nu,\nu)=(\nabla_\nu\nabla_\nu k)(\nu,\nu ) -\nabla_\nu\nabla_\nu \tr k $.
We denote this variation by 
\begin{equation}
   \frac{d }{d s}\frac{d }{d t}\left( \int_\Sigma P^2 \,d\mu \right):= \int_\Sigma 2 \mathcal{D}W_2(\alpha, \Tilde{\alpha} ) \,d\mu, 
\end{equation}
where we denote $\mathcal{D}W_2(\alpha, \Tilde{\alpha} ) :=  \frac{d }{d s} (W_2 \alpha)+ \alpha \Tilde{\alpha}H^2P^2  $.

We follow the notation of \cite{eichko}. Given $\xi \in \mathbb{R}^3$ and $r>0$ we denote by $S_{\xi, r}$ the sphere  
$$S_{\xi, r}= S_r (r \xi)= \{ 
x \in \mathbb{R}^3 : |x-r \xi|= r   \}$$
Given $u \in \mathcal{C}^4(S_{\xi, r}) $, we define the map
$$\Phi^u_{\xi, r}:S_{\xi, r} \longrightarrow \mathbb{R}^3 : x\mapsto  \Phi^u_{\xi, r}(x) =x +u(x)(r^{-1} x - \xi), $$

where  $r^{-1} x - \xi$ is the Euclidean unit normal to $ S_{\xi, r}$, and we denote by $\Sigma_{\xi, r }(u)= \Phi^u_{\xi, r}(S_{\xi, r} )$ the Euclidean graph of $u$ over $S_{\xi, r}$.

We want to  generalize to the initial data set setting  the  foliation constructed by Eichmair and Koerber in the following theorem.
\begin{theorem}[{\cite[Theorem 5]{eichko}}]    
     Let $(M, g)$ be $\mathcal{C}^4$-asymptotic to Schwarzschild with mass $m > 0$ and suppose that the scalar curvature $\mathrm{Sc}$ satisfies 
\[
x^i \partial_i\left(|x|^2 \mathrm{Sc}\right) \leq o(|x|^{-2})\quad  \text{and } \quad \mathrm{Sc}(x) - \mathrm{Sc}(-x) = o(|x|^{-4}).
\]
There exists a compact set $K \subset M$, a number $\lambda_0 > 0$, and stable, on-center, area-constrained critical Willmore spheres  $\Sigma(\lambda)$, $\lambda \in (0, \lambda_0)$,  with parameter $\lambda$ such that $M \setminus K$ is foliated by the family $\{\Sigma(\lambda) : \lambda \in (0, \lambda_0)\}$. 
 \end{theorem}
Our foliation construction follows the strategy of Eichmair–Koerber \cite{eichko}, in particular their use of a Lyapunov–Schmidt reduction. It differs from the Lyapunov–Schmidt reduction of the local foliations in \cite{diaz2023local} in that we project the equation only once, directly onto the complement of the kernel, rather than using an iterative projection scheme. Below, we give a brief outline of the construction in \cite{eichko}.

$ i)$ \textbf{Perturbation of spheres.} 
    Begin with the large coordinate spheres
    $S_r(r\xi)$, $ |\xi|<1$. An application of the implicit function theorem produces a nearby surface \(\Sigma_{r,\xi}\) of fixed area \(4\pi r^2\) which satisfies the area‐constrained Willmore  equation up to the span of the first spherical harmonics.

$ii)$ \textbf{Reduction to a finite‐dimensional problem.} 
    It is shown that \(\Sigma_{r,\xi}\) genuinely solves the constrained Willmore equation exactly if and only if \(\xi\) is a critical point of the  function
    \[
      G_r(\xi)
      = \lambda^2\Bigl(\!\int_{\Sigma_{r,\xi}}H^2\,d\mu - 16\pi + 64\pi\lambda^{-1}\Bigr).
    \]

$iii)$ \textbf{Convexity and existence of a minimizer.} 
    A detailed expansion of \(G_r\) in the asymptotic regime reveals it to be strictly convex for large \(r\), with a unique minimizer \(\xi(r)\) close to the origin.  The surfaces \(\Sigma_{r,\xi(r)}\) then assemble into a global foliation by Willmore spheres of the end.

$iv)$ \textbf{Uniqueness of the foliation.} 
    Finally, uniqueness follows by combining the uniqueness clause in the Implicit Function Theorem with the fact that for each \(r\) the only critical point of \(G_r\) is \(\xi(r)\).

In our setting we follow the same overall strategy, but now keep track of the extra terms involving the second fundamental form $k$.  A careful check of each step in \cite{eichko} shows that their foliation continues to exist   if one relaxed  the conditions on the scalar curvature by changing the little $o$ bounds by a big $\mathcal{O}$ bound with a small constant $\eta_i$.  Therefore, in our results  we consider such  $\mathcal{O}$-type decays.

Let $\alpha \in (0, 1)$ and $\gamma > 1$. We denote by $\mathcal{G}_1$ the space of $\mathcal{C}^{3,\alpha}$-Riemannian metrics and by $\mathcal{G}_2$ the space of $\mathcal{C}^{3,\alpha}$ symmetric 2-tensors on
$$\{y \in \mathbb{R}^3: \gamma^{-1} \leq |y| \leq \gamma\}$$

with the $\mathcal{C}^{3,\alpha}$-topology. Let $\Lambda_0(S_1(0))$ and $\Lambda_1(S_1(0))$ be the constants and first spherical harmonics
viewed as subspaces of $\mathcal{C}^{4,\alpha}(S_1(0))$, respectively. We use the symbol $\bot$ to denote the $L^2(S_1(0))$-orthogonal complements of these spaces. Now we derive the $k\neq 0$ generalization of the result \cite[Lemma 16]{eichko}.
\begin{lemma}
There exist open neighborhoods $\mathcal{U}_1$ of $\delta \in  \mathcal{G}_1$, $\mathcal{U}_2$ of $0 \in  \mathcal{G}_2$,  $\mathcal{V}$ of $0 \in \Lambda_1(S_1(0))^\bot$, and $I$ of $0 \subset \mathbb{R}$ as
well as smooth maps $u : \mathcal{U}_1 \times \mathcal{U}_2 \longrightarrow \mathcal{V}$ and $\lambda : \mathcal{U}_1 \times \mathcal{U}_2 \longrightarrow I$ such that for every $(g,k)\in\mathcal U_1\times\mathcal U_2$ the surface $\Sigma_{0,1}(u(g,k))$ has area equal to $4 \pi$ and satisfies
\begin{equation}\label{eulag1}
 \lambda H  +\Delta^\Sigma H + H|\mathring{B}|^2+ H \mathrm{Ric}(\nu, \nu)+P( \nabla_\nu \tr k - \nabla_\nu k(\nu,\nu )) - 2\diver_\Sigma (Pk(\cdot, \nu)) +\frac{1}{2}H P^2 \in \Lambda_1(S_1(0)).
\end{equation}  
Here, all geometric quantities are computed with respect to the surface $\Sigma_{0,1}(u(g))$, the metric $g$, and the tensor $k$. Moreover, if $(g_0,k_0) \in \mathcal{U}_1 \times \mathcal{U}_2 $,  $u_0 \in \mathcal{V} $ and $\lambda_0 \in I$ are such that $\Sigma_{0,1}(u_0)$ satisfies (\ref{eulag1}) with respect to $g_0$ and $k_0$, and has area equal to $4 \pi$, then $u_0 = u(g_0, k_0)$ and $\lambda_0 = \lambda (g_0, k_0)$.
\end{lemma}
\begin{proof}
The argument follows the same general scheme as in the time-symmetric case (\(k=0\)),
but since it constitutes a central step in the analysis we present a complete version here for clarity.

   Let $\Tilde{\mathcal{V}} $ be a neighborhood of $0 \in \Lambda_1(S_1(0))^\bot $,  $\Tilde{\mathcal{U}}_1 $ a neighbourhood of the Euclidean metric $\delta \in  \mathcal{G}_1$, and $\Tilde{\mathcal{U}}_2 $ a neighborhood of $0 \in  \mathcal{G}_2$, such that the map 
   \begin{equation}
       T: \Tilde{\mathcal{V}} \times \mathbb{R} \times \Tilde{\mathcal{U}}_1 \times \Tilde{\mathcal{U}}_2   \longrightarrow \Lambda_1(S_1(0))^\bot \times  \mathbb{R}
   \end{equation}
   given by 
   \begin{equation}
       T(u,\lambda, g, k)= \left( P_{\Lambda_1(S_1(0))^\bot}(W_1(g)+H \lambda +W_2(g,k)), |\Sigma|   \right)
   \end{equation}
where $P_{\Lambda_1(S_1(0))^\bot}$ denotes the orthogonal projection to $ \Lambda_1(S_1(0))^\bot$, is well-defined and smooth.
   
The derivative in the direction of the first variable at point $(0,0,\delta, 0) $ is 
$$(dT)_{(0,0,\delta, 0)} ( (u,0, 0,0)) = (  -\Delta^{\mathbb{S}^2} (- \Delta^{\mathbb{S}^2} -2)u   ,\,  8 \pi P_{\Lambda_0(S_1(0))} u  ) $$
and the derivative in the direction of the second variable at point $(0,0,\delta, 0) $ is 
$$(dT)_{(0,0,\delta, 0)} ( (0,\lambda, 0,0)) = (2 \lambda, 0 ) $$
Then the differential $dT$ is 
\begin{equation}
dT(u,\lambda) = 
    \begin{pmatrix}
 -\Delta^{\mathbb{S}^2} (- \Delta^{\mathbb{S}^2} -2)  & 2  \\
 8 \pi P_{\Lambda_0(S_1(0))}  & 0 
\end{pmatrix}
\begin{pmatrix}
u \\
\lambda \\
\end{pmatrix}
\end{equation}
If we manage to see that $dT$ is an isomorphism, then by the implicit function theorem, we would have the result.  Now to see the $dT$ is injective, suppose that there are $u_1,u_2 \in  \Tilde{\mathcal{V}} \subset \Lambda_1(S_1(0))^\bot $ and $ \lambda_1, \lambda_2 \in \mathbb{R} $ such that 
\begin{equation}
    \begin{pmatrix}
 -\Delta^{\mathbb{S}^2} (- \Delta^{\mathbb{S}^2} -2)  & 2  \\
 8 \pi P_{\Lambda_0(S_1(0))}  & 0 
\end{pmatrix}
\begin{pmatrix}
u_1 \\
\lambda_1 \\
\end{pmatrix}
=
\begin{pmatrix}
 -\Delta^{\mathbb{S}^2} (- \Delta^{\mathbb{S}^2} -2)  & 2  \\
 8 \pi P_{\Lambda_0(S_1(0))}  & 0 
\end{pmatrix}
\begin{pmatrix}
u_2 \\
\lambda_2 \\
\end{pmatrix}
\end{equation}
this implies 
$$  8 \pi P_{\Lambda_0(S_1(0))}u_1=  8 \pi P_{\Lambda_0(S_1(0))}u_2, \quad    -\Delta^{\mathbb{S}^2} (- \Delta^{\mathbb{S}^2} -2) u_1 +2\lambda_1=  -\Delta^{\mathbb{S}^2} (- \Delta^{\mathbb{S}^2} -2) u_2 +2\lambda_2  $$

Combining the first equation with the second we obtain.
$$   -\Delta^{\mathbb{S}^2} (- \Delta^{\mathbb{S}^2} -2) (\Tilde{u}_1- \Tilde{u}_2) = 2(\lambda_2 -\lambda_1) $$
where we denote $\Tilde{u}:=P_{(\Lambda_0(S_1(0)) \times \Lambda_1(S_1(0)))^\bot}(u) $ the orthogonal projection of $u$ to $(\Lambda_0(S_1(0)) \times \Lambda_1(S_1(0)))^\bot$, we also denote $\Lambda_i(S_1(0))$  by $\Lambda_i$ .

 Since $-\Delta^{\mathbb{S}^2} (- \Delta^{\mathbb{S}^2} -2):(\Lambda_0 \times \Lambda_1)^\bot \to (\Lambda_0 \times \Lambda_1)^\bot $ is an isomorphism and $\Tilde{u}_1- \Tilde{u}_2 \in (\Lambda_0 \times \Lambda_1)^\bot $, then $2(\lambda_2 -\lambda_1) \in (\Lambda_0 \times \Lambda_1)^\bot$, and with this we have  that $\lambda_1= \lambda_2$ and $u_1 =u_2 $. 

 Now to show that  $dT$ is surjective  consider a given  $u_1 \in  \Tilde{\mathcal{V}} \subset \Lambda_1^\bot $ and $ \lambda_1 \in \mathbb{R} $.  By choosing  $u_2 \in  \Tilde{\mathcal{V}} \subset \Lambda_1^\bot $ and $ \lambda_2 \in \mathbb{R} $ such that  $\lambda_2 = \frac{1}{2} P_{\Lambda_0}(u_1) $,  $P_{(\Lambda_0 \times \Lambda_1)^\bot}(u_2)=(-\Delta^{\mathbb{S}^2} (- \Delta^{\mathbb{S}^2} -2))^{-1} P_{(\Lambda_0 \times \Lambda_1)^\bot}(u_1) $ and  $P_{\Lambda_0}(u_1)=P_{\Lambda_0}(u_1)$, it holds $dT(u_2,\lambda_2)= (u_1,\lambda_1)^T$. Then $dT$ is an isomorphism and by the implicit function theorem we have the result.
\end{proof}
We denote $ \Lambda_0(S_{\xi,r})$ 
and $ \Lambda_1(S_{\xi,r})$ to be the constants and first spherical harmonics viewed as subspaces of  $\mathcal{C}^{4,\alpha}(S_{\xi, r})$, respectively.

Consider the map $ \Theta_{\xi, r}(y)= r( \xi +y ) $,  like in \cite{diaz2023local} or \cite{main1local}, where local foliations in initial data sets were constructed,  we consider rescaled data    $g_{\xi, r} = r^{-2} \Theta^*_{\xi, r} g$ and $ k_{\xi, r} = r^{-1} \Theta^*_{\xi, r} k$ on $M$. Then on this rescaled manifold, we define the function
\begin{equation}\label{resc}
\begin{split}
    \Psi(r, \xi , \lambda, u)=&r^2 \lambda H_{r,\xi}  +\Delta_{r,\xi}^\Sigma H_{r,\xi} + H_{r,\xi}|\mathring{B}_{r,\xi}|^2+ H_{r,\xi} \,\mathrm{Ric}_{r,\xi}(\nu_{r,\xi}, \nu_{r,\xi})+\frac{1}{2}H_{r,\xi} P_{r,\xi}^2\\&+P_{r,\xi}( \nabla_{\nu_{r,\xi}} \tr k_{r,\xi} - \nabla_{\nu_{r,\xi}} k_{r,\xi}(\nu_{r,\xi},\nu_{r,\xi} ))  - 2\diver_\Sigma (P_{r,\xi}  k_{r,\xi}(\cdot, \nu_{r,\xi}))
    \end{split}
\end{equation}
where the right hand side is evaluated on $\Sigma_{\xi, r }(u)= \Phi^u_{\xi, r}(S_{\xi, r} )$ with respect to $g_{\xi,r}$  and $k_{\xi, r}$.  Since  $g_{\xi,r}$ and $k_{r,\xi}$  are conformal to $g$ and $k$ in the coordinates given by  $ \Theta_{\xi, r}$  and using how the different terms on  (\ref{resc}) change under this conformal transformation (for instance, $H_{r,\xi}= r H$, $\nu_{r,\xi}= r \nu $, $P_{r,\tau}= r P $ etc) one obtains  $ \Psi(r, \xi , \lambda, u)= r^3 (W_1+W_2 +\lambda H)(\Theta^*_{\xi, r} g, \Theta^*_{\xi, r} k, \Sigma)$. Hence,  if $\Psi(r, \xi , \lambda, u)= r^3 (W_1+W_2+\lambda H)(\Theta^*_{\xi, r} g, \Theta^*_{\xi, r} k, \Sigma) \in \Lambda_1(S_1(0)) $, then $(W_1+W_2+\lambda H)(g,  k, \Sigma) \in \Lambda_1(S_{\xi,r}) $. In particular,  by the previous lemma, we have:

\begin{proposition}\label{17EK}
    There  are constants $r_0 > 1$, $c > 1$, and $\epsilon > 0$ depending on $g$, $k$ and $\Tilde{\delta} \in (0, 1/2)$ such that for every $ \xi \in \mathbb{R}^3$ with $|\xi| < 1 - \Tilde{\delta} $ and every $r > r_0$ there exist \(u_{\xi,r} \in \mathcal{C}^{\infty}(S_{\xi,r})\)
and \(\lambda_{\xi,r} \in \mathbb{R}\) such that the following hold. The surface
\(\Sigma_{\xi,r} = \Sigma_{\xi,r}(u_{\xi,r})\)
has the properties
\[
W_1 +W_2+ \lambda_{\xi,r} H \in \Lambda_1(S_{\xi,r}) \quad \text{and} \quad |\Sigma_{\xi,r}| = 4 \pi r^2.
\]
There holds \(u_{\xi,r} \perp \Lambda_1(S_{\xi,r})\) and
\[
|u_{\xi,r}| + r | ^{S_{\xi,r}}\nabla u_{\xi,r}| + r^2 |^{S_{\xi,r}}\nabla^2 u_{\xi,r}| + r^3 |^{S_{\xi,r}}\nabla^3 u_{\xi,r}| + r^4 |^{S_{\xi,r}}\nabla^4 u_{\xi,r}| < c, \quad r^3 |\lambda_{\xi,r}| < c
\]
Moreover, if \(\lambda \in \mathbb{R}\) and \(\Sigma_{\xi,r}(u)\) with \(u \perp \Lambda_1(S_{\xi,r})\) are such that
\[
W_1 +W_2+ \lambda H \in \Lambda_1(S_{\xi,r}), \quad |\Sigma_{\xi,r}(u)| = 4 \pi r^2,
\]
and
\[
|u| + r |^{S_{\xi,r}}\nabla u| + r^2 |^{S_{\xi,r}}\nabla^2 u| + r^3 |^{S_{\xi,r}}\nabla^3 u| + r^4 |^{S_{\xi,r}}\nabla^4 u| < \epsilon r,\quad r^3 |\lambda| < \epsilon r,
\]
then \(u = u_{\xi,r}\) and \(\lambda = \lambda_{\xi,r}\). Furthermore it holds that \[
(\Bar{D}u)|_{(\xi, r)} = \mathcal{O} (r^{-1}) \quad \text{and} \quad u'|_{(\xi, r)} = \mathcal{O} (r^{-2}),
\]
where \(\Bar{D}\) and the dash indicate differentiation with respect to the parameters \(\xi\) and \(r\), respectively.
\end{proposition}
We denote  \(u_{\xi,r}\) by \(u\), \(\lambda_{\xi,r}\) by \(\lambda\), and \(\Lambda_\ell(S_{\xi,r})\) by \(\Lambda_\ell\)
for \(\ell = 0, 1\). Now just as in \cite[Lemma 19]{eichko} we obtain.
\begin{lemma}
    There holds 
    $P_{\Lambda_0} u = -2 + \mathcal{O}(r^{-1}) $  if  $|\xi|<1- \Tilde{\delta}$.
\end{lemma}
As in \cite[Lemma 20]{eichko} we expand $u$ in spherical harmonics.   
\begin{lemma}\label{expansquafoli}
    If \(|\xi| < 1 - \Tilde{\delta}\), there holds
\[
\lambda = 4 r^{-3} + \mathcal{O} (r^{-4}),
\]
\[
W_1(\Sigma_{\xi,r}) + W_2(\Sigma_{\xi,r})+ \lambda H(\Sigma_{\xi,r}) = \mathcal{O} (r^{-5}),
\]
\[
u = -2 + 4 \sum^{\infty}_{\ell= 2} \frac{|\xi|^\ell}{\ell} P_\ell \left(-|\xi|^{-1} \delta (y, \xi) \right) + \mathcal{O}(r^{-1}).
\]
Where $P_\ell$ are Legendre Polynomials (see \cite[Appendix B]{eichko}).
\end{lemma}
\begin{proof}
    Note that $W_2(\Sigma_{\xi,r}) = \mathcal{O}(r^{-5})$, then the proof follows roughly as in the case $k=0$ of \cite[Lemma 20]{eichko}.
    
   We denote by $^\delta\Delta_{\xi,r}$ the  laplacian on $S_{\xi,r}$ with respect to the Euclidean metric $\delta$.   We have by the decay of the metric 
    $$W_1(\Sigma_{\xi,r}) -W_1(S_{\xi,r}) = - \,^\delta\Delta^2_{\xi,r} u-\frac{2}{r^2}\, ^\delta\Delta_{\xi,r}u +\mathcal{O}(r^{-5})  $$
    and $$ H= \frac{2}{r} +\mathcal{O}(r^{-2}) $$
    By Proposition \ref{17EK}
$$W_1(\Sigma_{\xi,r}) + W_2(\Sigma_{\xi,r})+ \lambda H(\Sigma_{\xi,r}) =Y_1 \in \Lambda_1 .$$
Then, using that by the decay of $k$, $W_2(\Sigma_{\xi,r})=\mathcal{O}(r^{-5})$ and
\begin{equation}
    \begin{split}
      ^\delta \Delta^2_{\xi,r} u+\frac{2}{r^2}\, ^\delta\Delta_{\xi,r} u  -\frac{2}{r}\lambda &= W_1(S_{\xi,r})-W_1(\Sigma_{\xi,r})  -\frac{2}{r}\lambda  +\mathcal{O}(r^{-5})  \\
     &=  W_1(S_{\xi,r})-W_1(\Sigma_{\xi,r}) -H(\Sigma_{\xi,r}) \lambda +\mathcal{O}( \lambda r^{-2})  +\mathcal{O}(r^{-5}) \\
     &= W_1(S_{\xi,r}) - Y_1 + \mathcal{O}(r^{-5})\\
    \end{split}
\end{equation}
The rest of the proof follows as for the $k=0$ case. 
\end{proof}
We consider the following modification to the  Hawking functional. 
\begin{equation}
    F_r(\Sigma ):=r^2 \left( \int_\Sigma H^2 -P^2 d\mu -16 \pi +64\pi r^{-1} \right)
\end{equation}
and we  define the function
\[
G_r : \{\xi \in \mathbb{R}^3 : |\xi| < 1 - \Tilde{\delta} \} \to \mathbb{R} \quad \text{given by} \quad G_r(\xi) = F_r(\Sigma_{\xi,r}).
\]
Then as in \cite[Lemma 21]{eichko}, it is  direct to see that:
\begin{lemma}\label{lemma21}
    There is \(r_0 > 1\) depending on \(g\), $k$ and \(\Tilde{\delta} \in (0, 1/2)\) with the following property. Let
\(r > r_0\). Then \(\Sigma_{\xi,r}\) is an area-constrained critical sphere of the Hawking functional if and only if \(\xi\) is a critical point of \(G_r\).
\end{lemma}
With this result, finding critical surfaces of the Hawking energy that are deformations of coordinate spheres reduces to identifying critical points of the function $G_r$. To locate such critical points, we require a more explicit form of $G_r$. In \cite[Lemma 22]{eichko}, the asymptotic expansion of 
$G_r$ was computed in the absence of the tensor $k$. In our case, a similar expansion is derived using the same procedure, but with an additional term accounting for $k$.
\begin{lemma}
    If \(|\xi| < 1 - \Tilde{\delta}\), there holds
\[
\begin{split}
G_r(\xi) =& 64 \pi + \frac{32 \pi}{1 - |\xi|^2} - 48 \pi |\xi|^{-1} \log \left( \frac{1 + |\xi|}{1 - |\xi|} \right) - 128 \pi \log(1 - |\xi|^2) + 2 r \int_{\mathbb{R}^3 \setminus B_r(r \xi)} \mathrm{Sc} \, dv_\delta \\
&-r^2 \int_{S_r(r \xi)} \pi(\nu,\nu)^2 d\mu_\delta +  \mathcal{O} (r^{-1}).
\end{split}
\]
where  $\pi= k-\tr k \, g $ is the canonically conjugated momentum and the trace is taken with respect to the Euclidean metric.
\end{lemma}
We have also that $G_r$ is bounded:
\begin{lemma}\label{bounded}
    There are \(c > 0\) and \(r_0 > 1\) which depend only on \(g\), $k$ and \( \hat{\delta} \in (0, 1/2)\) such that
\[
\|G_r\|_{\mathcal{C}^3}(\{\xi \in \mathbb{R}^3 : |\xi| \leq 1 - \hat{\delta} ) < c
\]
for every \(r > r_0\).
\end{lemma}
Now we can obtain the following result, which is analogous to \cite[Lemma 24]{eichko}. 
\begin{lemma}\label{convexity}
There exist constants $ \Tilde{\eta}_i>0$, i=1,2,3 such that if for $\eta_i< \Tilde{\eta}_i $ it holds 
\begin{equation}
    x^i \partial_i (|x|^2 \mathrm{Sc}) \leq \mathcal{O} (\eta_1 |x|^{-2}),\quad  \mathrm{Sc}(x)-\mathrm{Sc}(-x) = \mathcal{O} (\eta_2 |x|^{-4})
\end{equation}
and 
\begin{equation}
k_{ij}(x) = \mathcal{O}(\eta_3 |x|^{-2}) , \quad (\nabla k)_{ij}(x)= \mathcal{O}(\eta_3 |x|^{-3}). 
\end{equation}
then there exist
\(\tau > 0\), \(\Tilde{\delta}_0 \in (0, 1/2)\), and \(r_0 > 1\), $\Tilde{\eta}_i>0$ for $i=1,2,3$ depending only on \(g\) and $k$ such that, provided \(r > r_0\) and $\Tilde{\eta}_i> \eta_i $,
\[
\Bar{D}^2 G_r \geq \tau \, \text{Id}
\]
holds on \(\{\xi \in \mathbb{R}^3 : |\xi| < \Tilde{\delta}_0\}\), where \(\Bar{D}\) denotes differentiation with respect to the parameter \(\xi\). Moreover, given \(\Tilde{\delta} \in (0, 1/2)\) and \(\Tilde{\delta}_1 \in (0, 1 - \Tilde{\delta})\), there is \(r_1 > r_0\) such
that \(G_r\) is strictly increasing in radial directions on \(\{\xi \in \mathbb{R}^3 : \Tilde{\delta}_1 < |\xi| < 1 - \Tilde{\delta}\}\) provided \(r > r_1\).
\end{lemma}
\begin{proof}
    We write
\[
G_r = G_1 + G_{r,2}+G_{r,3} + \mathcal{O} (r^{-1})
\]
where
\[
G_1(\xi) = 64 \pi + \frac{32 \pi}{1 - |\xi|^2} - 48 \pi |\xi|^{-1} \log \left( \frac{1 + |\xi|}{1 - |\xi|} \right) - 128 \pi \log(1 - |\xi|^2),
\]
\[
G_{r,2}(\xi) = 2 r \int_{\mathbb{R}^3 \setminus B_r(r \xi)} \mathrm{Sc} \, dv_\delta.
\]
and 
$$G_{r,3}(\xi)= -r^2 \int_{S_r(r \xi)} \pi(\nu,\nu)^2 d\mu_\delta.  $$
By using the same argument used in \cite[Lemma 24]{eichko} to estimate $\Bar{D}_\xi G_{r,2}$ we obtain
 \begin{equation}\label{tayexpsc}
       |\xi|^{-1} \xi^i (\partial_i G_{r,2})(\xi) = -2 r^2 \int_{S_{\xi, r}} |\xi|^{-1} \delta( \xi, \nu ) \mathrm{Sc} \, d\mu_\delta\geq -\mathcal{O}(\eta_1) -\mathcal{O}(\eta_2) 
 \end{equation}
 Like in (\ref{firstvary}) we have 
\begin{equation}
\begin{split}
    |\xi|^{-1} \xi^i (\partial_i G_{r,3})(\xi) =& - r^2 \int_{S_{\xi,r}} 2 rP  \delta( \frac{\xi}{|\xi|}, \nu ) ( \nabla_{\nu} \tr k - \nabla_{\nu} k(\nu,\nu )) +4r P    k(\nabla^{S_{\xi,r}} \delta( \frac{\xi}{|\xi|}, \nu ), \nu )\\
   &+ r \delta( \frac{\xi}{|\xi|}, \nu ) H P^2  d\mu_\delta\\
   =&-r^2 \int_{S_{\xi,r}} 2 \pi(\nu, \nu) r \delta( \frac{\xi}{|\xi|}, \nu ) \nabla_{\nu}\pi(\nu,\nu) +4 \pi(\nu, \nu)    k( \frac{\xi}{|\xi|} , \nu )\\
   &+ 2 \delta( \frac{\xi}{|\xi|}, \nu )  \pi(\nu, \nu)^2  d\mu_\delta\\
   =&-r^2 \int_{S_{r}(\xi r)}  2 W_2 \alpha d\mu_\delta
   \end{split}
\end{equation}
with $ \alpha= r \delta( \frac{\xi}{|\xi|}, \nu ) $ and where since $\pi = k-\tr k g$, $P=-\pi(\nu, \nu)$.

Now we consider a Taylor expansion around the origin of the integral expression, that is, we consider the function $f(\xi)=-r^2 \int_{S_{r}(\xi r)}  2 W_2 \alpha d\mu_\delta $ Then, applying Taylor’s theorem at $\xi=0$, we obtain
\begin{equation}\label{taylorexp}
\begin{split}
  |\xi|^{-1} \xi^i (\partial_i G_{r,3})(\xi) =& -r^2 \int_{S_{r}(0)}   2 W_2 \alpha d\mu_\delta- \sum_{i=1}^{3} \partial_i \left( r^2\int_{S_{r}(\xi r)}  2 W_2 \alpha d\mu_\delta\right)_{|\xi=0} \xi^i r \\
 &- \frac{r^2}{2} \Bar{D}^2 \left(\int_{S_{r}(\xi r)}  2 W_2 \alpha d\mu_\delta\right)_{|\xi = \Tilde{\xi}}(r \xi, r \xi).\\
 \end{split}
\end{equation}
where $\tilde{\xi}$ is a point on the segment between $0$ and $\xi$ (i.e.\ $\tilde{\xi}=t\xi$ for some $t\in(0,1)$), and by (\ref{secondvary}), the second term of the integral can be written as 
\begin{equation}
     -  r^2\int_{S_{r}(0)}   2 \mathcal{D}W_2\left(r \delta( \frac{\xi}{|\xi|}, \nu ), r \delta( \xi, \nu ) \right) d\mu_\delta .
\end{equation}
Recall that $\nu= x r^{-1}- \xi = \rho \frac{|x|}{r}- \xi$ on $S_r(r \xi)$, where $\rho= \frac{x}{|x|}$, then using the decay of $k$  and that  $\frac{1}{|x|} \leq \frac{1}{r|1-|\xi||}$ on $S_r(r \xi)$, we have 
\begin{equation}\label{estitaylor}
    \begin{split}
      -r^2 \int_{S_{r}(0)}   2 W_2 \alpha d\mu_\delta&=-  \mathcal{O}(\eta_3^2) \\
      -  r^2\int_{S_{r}(0)}   2 \mathcal{D}W_2\left(r \delta( \frac{\xi}{|\xi|}, \nu ), r \delta( \xi, \nu ) \right) d\mu_\delta &= -  \mathcal{O}(\eta_3 |\xi|)\\
      - \frac{r^2}{2} \Bar{D}^2 \left(\int_{S_{r}(\xi r)}  2 W_2 \alpha d\mu_\delta\right)_{|\xi = \Tilde{\xi}}(r \xi, r \xi)&= -  \mathcal{O}(\eta_3 |\xi|^2).
    \end{split}
\end{equation}
 Note that all the terms except the last one depend on the radial directions,  $\pi(\rho,\rho)$, $(\nabla_\rho \pi)(\rho,\rho)$ and $k(\rho,\cdot) $. We have that (\ref{taylorexp}) give us 
\begin{equation}\label{varG3}
\begin{split}
    |\xi|^{-1} \xi^i (\partial_i G_{r,3})(\xi) \geq& -\mathcal{O}(\eta_3^2) -\mathcal{O}(\eta_3 |\xi|). 
   \end{split}
\end{equation}
 One can see that $G_1$ is strictly convex and   strictly  increasing in radial direction with $G_1(0)=0$ and   
\[
|\xi|^{-1}  \xi^i (\partial_i G_1)(\xi )\geq 256 \pi |\xi|.
\]
 Putting everything together we have 
\begin{equation}\label{DGexp}
    \Bar{D}_\xi G_r = \Bar{D}_\xi G_1 + \Bar{D}_\xi G_{r,2}+\Bar{D}_\xi G_{r,3} \geq  256 \pi |\xi| -\mathcal{O}(\eta_1) -\mathcal{O}(\eta_2) -\mathcal{O}(\eta_3) - \mathcal{O}(\eta_3|\xi|).
\end{equation}
Consequently, there exists  \(\tilde\delta = \tilde\delta(\eta_i) > 0\) such that, if each \(\eta_i\) is sufficiently small and 
$\tilde\delta < |\xi| \ll 1$, then 
$\Bar{D}_\xi G_r(\xi) > 0$.

For convexity, as in \cite[Lemma 24]{eichko} one sees that 
\begin{equation}
    \Bar{D}^2_\xi G_{r,2} \geq -( \mathcal{O}(\eta_1 ) +c |\xi|) \text{Id}
\end{equation}
Now for $\Bar{D}^2_\xi G_{r,3}$ we need to consider the second variation (\ref{secondvary}) with $\alpha= \Tilde{\alpha}= r \delta( \frac{\xi}{|\xi|}, \nu )$,  then we have
\begin{equation}
    \Bar{D}^2_\xi G_{r,3} \geq - \mathcal{O}(\eta_3 ). 
\end{equation}
 As $G_{r,1}$ is strictly convex we have that for $ \eta_i$, $|\xi|$  small enough 
\begin{equation}
   \Bar{D}^2_\xi G_{r} \geq \tau \text{Id} 
\end{equation}
for some $\tau>0$.
\end{proof}

Now we can state the existence result for the foliation.
\begin{theorem}\label{exisfoli}
    Let $(M,g,k)$ be $ \mathcal{C}^4$-asymptotic to Schwarzschild with mass $m>0$. There exist constants $ \Tilde{\eta}_i>0$, i=1,2,3 such that if for $\eta_i< \Tilde{\eta}_i $ it holds 
\begin{equation}\label{assumptionsscalar}
    x^i \partial_i (|x|^2 \mathrm{Sc}) \leq \mathcal{O} (\eta_1 |x|^{-2}),\quad  \mathrm{Sc}(x)-\mathrm{Sc}(-x) = \mathcal{O} (\eta_2 |x|^{-4})
\end{equation}
and 
\begin{equation}
k_{ij}(x) = \mathcal{O}(\eta_3 |x|^{-2}) , \quad (\nabla k)_{ij}(x)= \mathcal{O}(\eta_3 |x|^{-3}). 
\end{equation}
 Then there exists a compact set \(K \subset M\), a number \(\lambda_0 > 0\), and on-center stable area-constrained critical
spheres \(\Sigma(\lambda)\) of the Hawking functional, satisfying (\ref{eulag}) with parameter  \(\lambda \in (0, \lambda_0)\), such that \(M \setminus K\) is foliated by the family
\(\{\Sigma(\lambda) : \lambda \in (0, \lambda_0)\}\).
\end{theorem}
\begin{proof}
  We follow the structure of the  proof in \cite{eichko}.  First, we show that for every \(r > 1\) sufficiently large, there exists a stable
area-constrained critical sphere with area radius \(r\).
To see this, we decompose
\[
G_r(\xi) = G_1(\xi) + 2 r \int_{\mathbb{R}^3 \setminus B_r(r \xi)} \mathrm{Sc} \, d{v}_\delta -r^2 \int_{S_r(r \xi)} \pi(\nu,\nu)^2 d\mu_\delta  + \mathcal{O} (r^{-1})
\]
 Note that \(G_1(0) = 0\) and \(\lim_{|\xi| \to 1} G_1(\xi) = \infty\).  Now by the first term in (\ref{assumptionsscalar}),    $\rho (| x|^2 \mathrm{Sc})(x) \leq  \mathcal{O}( \eta_1| x|^{-3})$, where $\rho= \frac{x^i}{|x|}\partial_i= \frac{\partial}{\partial r}$,  then we can integrate 
\begin{equation}\label{estilow}
    |x|^2 \mathrm{Sc} = - \int_{|x|}^\infty  \frac{\partial}{\partial s} (s^2 \mathrm{Sc}) ds \geq - \int_{|x|}^\infty  \mathcal{O}(\eta_1 s^{-3})  ds = -\mathcal{O}(|x|^{-2})
    \end{equation}
 and obtain that $ \mathrm{Sc} \geq -\mathcal{O}(\eta_1|x|^{-4}) $. With this we  we find,
\[
2 r \int_{\mathbb{R}^3 \setminus B_r(r \xi)} \mathrm{Sc} \, dv_\delta \geq -2 r \int_{\mathbb{R}^3 \setminus B_{r(1 - |\xi|)}(0)} \mathcal{O}(\eta_1|x|^{-4})  \, dv_\delta \geq -\mathcal{O}(\eta_1 |1 - |\xi||^{-1}).
\]
and similarly 
\begin{equation}
    -r^2 \int_{S_r(r \xi)} \pi(\nu,\nu)^2 d\mu_\delta  \geq -\mathcal{O}((\eta_3^2 +\eta_3^2|\xi| ) |1 - |\xi||^{-1})
\end{equation}
It follows that there is a number \(z \in (1/2, 1)\) such that
\[
G_r(0) < G_r(\xi)
\]
for every \(\xi\) with \(|\xi| = z\) and every sufficiently large \(r > 1\). Then by convexity, there is a local minimum
\(\xi(r)\) of \(G_r\) with
\[
|\xi(r)| < z.
\]
Lemma \ref{lemma21} shows that \(\Sigma(r) = \Sigma_{\xi(r), r}\) is a stable area-constrained critical surface, this comes since the stability operator of the surface is $Lu = Q_\Sigma u + \mathcal{O}(r^{-5})=  \,^\delta\Delta^2_{\xi,r} u+\frac{2}{r^2} ^\delta \Delta_{\xi,r} u + \mathcal{O}(r^{-5}) $, then as $^\delta\Delta^2_{\xi,r} +\frac{2}{r^2}\, ^\delta \Delta_{\xi,r}$ is a stable operator, $L$ is stable for $r$ large enough.

Note that since  \((\Bar{D} G_r)(\xi(r)) = 0\) and by  (\ref{DGexp}) we find
\[
0 = |\xi(r)|^{-1}  \xi(r)^i (\partial_i G_r)(\xi(r)) \geq    256 \pi |\xi| -\mathcal{O}(\eta_1) -\mathcal{O}(\eta_2) -\mathcal{O}(\eta_3) - \mathcal{O}(\eta_3|\xi|).
\]
and then we have 
\begin{equation}\label{estixi}
    \mathcal{O}(\eta_1)+ \mathcal{O}(\eta_2)+ \mathcal{O}( \eta_3) \geq 256 \pi |\xi| -C\eta_3|\xi| \geq C |\xi|
\end{equation}
provided that $256 \pi  -C\eta_3 >0 $, which is satisfied for $\eta_3$ sufficiently small.   We also have that $\lambda (r)\to 0$ as $r \to 0$. Now it remains to see that the surfaces form a foliation.

Let \(a \in \mathbb{R}^3\). Differentiating the identity \((\Bar{D} G_r)|_{\xi(r)}(a) = 0\) with respect to \(r\), we obtain
\begin{equation}\label{secodiffG}
    (\Bar{D}^2 G_r)|_{\xi(r)}(a, \xi^\prime (r)) + (\Bar{D} G^\prime_r)|_{\xi(r)} (a) = 0.
\end{equation}
 We need to estimate  $(\Bar{D}G^\prime_r)|_{\xi(r)} (a) $, then proceeding as in \cite[Proposition 49]{eichko} one finds 
 \begin{equation}
     (\Bar{D} G^\prime_{2,r})|_{\xi(r)} (a) = \mathcal{O}((\eta_1 +\eta_2) r^{-1})
 \end{equation}
 and by taking the $r$ derivative it is also direct to see
 \begin{equation}\label{dg3esti}
     (\Bar{D} G^\prime_{3,r})|_{\xi(r)} (a) = \mathcal{O}((\eta_3+\eta_3 |\xi| ) r^{-1})
 \end{equation}
 and therefore using the estimate on $|\xi|$ we obtain 
\begin{equation}
     (\Bar{D} G^\prime_r)|_{\xi(r)} (a) = \mathcal{O}((\eta_1 +\eta_2 +\eta_3) r^{-1}).
\end{equation}
Choosing $a= \xi'(r)  $, using (\ref{secodiffG}) and the convexity of $G_r$  we have for some $\tau>0$ 
$$\tau |\xi'(r)|^2\leq  (\Bar{D}^2 G_r)|_{\xi(r)}(\xi^\prime (r), \xi^\prime (r))  \leq - (\Bar{D} G^\prime_r)|_{\xi(r)} (\xi^\prime) = \mathcal{O}((\eta_1 +\eta_2 +\eta_3)|\xi'(r)| r^{-1}).$$

Then $|\xi'(r)|= \mathcal{O}((\eta_1 +\eta_2 +\eta_3)r^{-1})  $. Note that the parameterization of our surfaces is given by $ \Phi(y, r)= r y +r \xi +u_{\xi, r} (y) y$ for $ y\in S^2(0)$. then
$$ \frac{\partial \Phi}{\partial r} = y + \xi(r) +r \xi'(r )+ \Bar{D}_{\xi'(r)} u_{\xi, r} y+ u_{\xi, r} '(y) y  $$
then 
\begin{equation}
\begin{split}
    \delta( \frac{\partial \Phi}{\partial r}, y ) &=1 + \delta( \xi(r),y ) +r \delta( \xi'(r ),y) + \Bar{D}_{\xi'(r)} u_{\xi, r} + u_{\xi, r} '(y) \\
    &\geq 1 -|\xi(r)| + \mathcal{O}(\eta_1 +\eta_2 +\eta_3)  +\mathcal{O}( r^{-1})
    \end{split}
\end{equation}
Then for small $\eta_i$ and large $r$, we have a foliation.
\end{proof}
\begin{remark}
Note that compared to \cite{eichko} we relaxed  the assumptions on $\mathrm{Sc}$ by changing the little $o$ bounds by big $\mathcal{O}$ bound with a small constant $\eta_i$. By relaxing the assumption this way we obtain that   $\lim_{r \to \infty} \xi(r)$ is not necessarily zero, meaning that the surfaces are not centered and their center may slightly move while preserving the foliation property.  Also, the main difference with the assumptions for the foliation constructed in \cite{Friedrich2020} is that there the technique of Lamm, Metzger and Schulze \cite{willflat} was used. This requires stronger conditions on the decay of the metric and on the scalar curvature for which a decay of the form $ \mathrm{Sc}=\mathcal{O}(\eta |x|^{-5}) $ for small $\eta$ was required.
\end{remark}

We also obtain a uniqueness result with a proof analogous to the one of \cite[Theorem 8]{eichko} which we include for completeness.
\begin{theorem}\label{uniqueness}
    Under the same assumptions of Theorem \ref{exisfoli}. There exists a small constant \(\epsilon_0 > 0\) and a compact set
\(K \subset M\) which depend only on \((M, g, k)\) such that the following holds. For every \(\Tilde{\delta} > 0\), there exists a
large constant \(r_0 > 1\) such that every area-constrained critical sphere \(\Sigma \subset M \setminus K\) of the Hawking functional with
\[
|\Sigma| > 4 \pi r_0^2, \quad \quad \Tilde{\delta} r(\Sigma) < R(\Sigma), \quad \Tilde{\delta} R(\Sigma) < r(\Sigma)  \quad \text{and}\quad \int_{\Sigma} |\mathring{B}|^2 \, d\mu < \epsilon_0 
\]
 where $R(\Sigma):=\sup_{x \in \Sigma } |x|$ and $r(\Sigma)$ is the area radius, belongs to the foliation from Theorem \ref{exisfoli}.
\end{theorem}

\begin{proof}
The argument follows closely the proof of \cite[Theorem 8]{eichko}, with the curvature control replaced by the scale-invariant estimates proved in 
Eichmair--Koerber--Metzger--Schulze \cite{eichmair2024huisken}.

 Suppose, for a contradiction, that there is a sequence of area-constrained critical spheres \(\{\Sigma_j\}_{j=0}^\infty\) with
\[
\liminf_{j \to \infty} |\Sigma_j| = \infty, \quad \limsup_{j \to \infty} \int_{\Sigma_j} |\mathring{B}|^2 \, d\mu < \epsilon_0, \quad 0 < \liminf_{j \to \infty} \frac{R(\Sigma_j)}{r(\Sigma_j)} \leq \limsup_{j \to \infty} \frac{R(\Sigma_j)}{r(\Sigma_j)} < \infty,
\]
and \(\Sigma_j \neq \Sigma(r_j)\) where \(r_j = r(\Sigma_j)\).

Since the tensor $k$ decays as $\mathcal{O}(|x|^{-2})$ and its derivatives as $\mathcal{O}(|x|^{-3})$, 
the Hawking operator differs from the Willmore operator by lower--order perturbations.
Hence the scale–invariant curvature bounds for large area–constrained Willmore spheres 
from \cite[Lemmas 6 and 8, Proposition 10, Corollary 14]{eichmair2024huisken} 
apply to the present situation with only lower–order error terms.   In particular, for each sufficiently large $j$, 
\begin{equation}\label{eq:EKMS-curv}
 |x| \big|B - r_j^{-1} g_{|\Sigma_j}\big|
  + |x|^2\,|\nabla B|
  \;=\;
  \mathcal{O}(r_j^{-1/2}+ \frac{1}{\min_{x\in\Sigma_j} |x|}),
\end{equation}
By our assumptions $|x|\simeq r_j$ on $\Sigma_j$.
Substituting these comparabilities into \eqref{eq:EKMS-curv} yields
\begin{equation}\label{eq:C2bounds}
  \|B\|_{L^{\infty}(\Sigma_j)} = \mathcal{O}(r_j^{-1}),
  \qquad
  \|\nabla B\|_{L^{\infty}(\Sigma_j)} = \mathcal{O}(r_j^{-2}),
\end{equation}
We can also get an estimate on the Lagrange parameter $\lambda(\Sigma_j)$. Integrating the Euler–Lagrange equation \eqref{eulag}  over $\Sigma_j$
and using $\int_{\Sigma_j}\Delta_{\Sigma_j} H\,d\mu=0$ and $\int_{\Sigma_j}\mathrm{div}_{\Sigma_j} X\,d\mu=0$, we obtain
\[
\lambda(\Sigma_j)\int_{\Sigma_j} H\,d\mu
= -\!\int_{\Sigma_j}\!\Big(H|\mathring{B}|^2 + H\,\mathrm{Ric}(\nu,\nu)
+ P(\nabla_\nu \operatorname{tr}k - \nabla_\nu k(\nu,\nu))
+ \tfrac12 H P^2 - 2\,k(\nabla^\Sigma P,\nu)\Big)\,d\mu.
\]
Using \eqref{eq:C2bounds}, $H=2r_j^{-1}+\mathcal{O}(r_j^{-2})$, $|P|=\mathcal{O}(r_j^{-2})$, $|\nabla P|=\mathcal{O}(r_j^{-3})$,
$\mathrm{Ric}=\mathcal{O}(r_j^{-3})$, and $|\Sigma_j|=4\pi r_j^2+\mathcal{O}(1)$, we can estimate the right hand side integral by $\mathcal{O}(r_j^{-3})$.
Since $\int_{\Sigma_j} H\,d\mu = 8\pi r_j + \mathcal{O}(1)\simeq c\,r_j$, it follows that
\begin{equation}\label{eq:lambda}
  \lambda(\Sigma_j) = \mathcal{O}(r_j^{-3}).
\end{equation}

Define the rescaled surfaces $\tilde{\Sigma}_j := r_j^{-1}\Sigma_j 
\subset (\mathbb{R}^{3}, \delta)$.
By \eqref{eq:C2bounds}, the family $\{\tilde{\Sigma}_j\}$ has uniformly 
bounded curvature and area, and hence (after passing to a subsequence) converges 
smoothly to a closed Willmore surface 
$\tilde{\Sigma} \subset \mathbb{R}^{3}$ satisfying
\[
  \int_{\tilde{\Sigma}} |\mathring{B}_{\delta}|^{2}\, d\mu_{\delta} 
  < \varepsilon_0, 
  \qquad |\tilde{\Sigma}| = 4\pi, 
  \qquad \frac{1}{4\pi} \int_{\tilde{\Sigma}} y\, d\mu_{\delta} = \xi_0.
\]
By  \cite[Theorem 2.7]{kuwert2001willmore}  (or also by the quantitative rigidity results of De Lellis–Müller \cite[Theorem 1]{de2005optimal} and \cite[Theorem 2]{de2006ac}), it follows that $\tilde{\Sigma}=S_{1}(\xi_0)$ is a round sphere.

For $j$ large, each $\Sigma_j$ is therefore a small normal graph over a 
coordinate sphere:
\[
  \Sigma_j = \Sigma_{\xi_j, r_j}(u_j)
  \quad\text{with}\quad 
  \|u_j\|_{C^{4}(S^{2})} \le \varepsilon,
\]
for some small $\varepsilon>0$ depending only on $\varepsilon_0$. Then for $j$ large,  $W_1(\Sigma_j) + W_2(\Sigma_j)+ \lambda H(\Sigma_j) =0\in \Lambda_1(S_{\xi,r_{j}})$ for certain $\xi$ and $|\Sigma_j|= 4\pi r_j^2$, then by the uniqueness statement of the Lyapunov-Schmidt reduction of Proposition \ref{17EK}, $\Sigma_j $ is one of the surfaces obtained there in the sense that  \(\Sigma_j = \Sigma_{\xi_j, r_j}\) for some \(\xi_j\), but as $\Sigma_{\xi_j, r_j} $ is an area-constrained critical surface, then $\xi_j$ is also the only critical point of \(G_{r_j}\). This implies that $ \xi_j =\xi(r_j)$ getting a contradiction.
\end{proof}

\begin{remark}
    Note that compared to the case in \cite{eichko} $\xi_0 $ is not equal to zero, but this does not compromise the uniqueness result since it only relies on the surfaces being a perturbation of coordinate spheres so that one can use the uniqueness of the Lyapunov-Schmidt reduction.
\end{remark}
We can relax the conditions on $\mathrm{Sc}$ and $k$, obtaining a family of surfaces at infinity, similar to \cite[Theorem 6]{willcen}, but not a foliation. This means that the surfaces may intersect with each other. 
\begin{theorem}\label{family}
     Let $(M,g,k)$ be $ \mathcal{C}^4$-asymptotic to Schwarzschild with mass $m>0$. There exist constants $ \Tilde{\eta}_i>0$, i=1,2 such that if for $\eta_i< \Tilde{\eta}_i $ it holds 
     \begin{equation}
         \mathrm{Sc} \geq -\mathcal{O} (\eta_1 |x|^{-4}). 
     \end{equation}
  Then there exists a compact set \(K \subset M\), a number \(\lambda_0 > 0\), and a family
\(\{\Sigma(\lambda) : \lambda \in (0, \lambda_0)\}\) of on-center stable area-constrained critical
spheres \(\Sigma(\lambda)\) of the Hawking functional,  satisfying (\ref{eulag}) with parameter \(\lambda\) such that as $\lambda \to 0$, $|\Sigma(\lambda)|\to \infty$ and   $ R(\lambda) \to \infty$, where $R(\lambda):=\sup_{x \in \Sigma(\lambda) } |x|$. 

If additionally it holds,
\begin{equation}\label{condiuni1}
    x^i \partial_i (|x|^2 \mathrm{Sc}) \leq \mathcal{O} (\eta_1 |x|^{-2}),
\end{equation}
\begin{equation}\label{condiuni2}
    \pi (\rho, \rho) = \mathcal{O} (\eta_2  | x|^{-2}), \quad k(\cdot, \rho) = \mathcal{O} ( \eta_2| x|^{-2})\quad \text{and} \quad (\nabla_{\rho} \pi )(\rho, \rho)= \mathcal{O} (\eta_2| x|^{-3}),
\end{equation}
where $\pi = k - \tr k\, g$ and $\rho= \frac{x^i}{|x|}\partial_i$. Then the family of surfaces is unique in the sense that there exist a small constant \(\epsilon_0 > 0\) and a compact set
\(K \subset M\) which depend only on \((M, g, k)\) such that  for every \(\Tilde{\delta} > 0\), there exists a large constant \(r_0 > 1\) such that every area-constrained critical sphere \(\Sigma \subset M \setminus K\) of the Hawking functional with
\[
|\Sigma| > 4 \pi r_0^2, \quad \quad \Tilde{\delta} r(\Sigma) < R(\Sigma), \quad \Tilde{\delta} R(\Sigma) < r(\Sigma), \quad \text{and}\quad \int_{\Sigma} |\mathring{B}|^2 \, d\mu < \epsilon_0 
\]
 where $r(\Sigma)$ is the area radius, belongs to the family of surfaces.\end{theorem}
\begin{proof}We use the decomposition 
\begin{equation}\label{Gr-split-proof}
G_r(\xi)=G_1(\xi)+G_{r,2}(\xi)+G_{r,3}(\xi)+O(r^{-1}),
\end{equation}
with
\[
G_{r,2}(\xi)=2r\int_{|x|>r|\xi|}\mathrm{Sc}(x)\,dv_\delta,\qquad
G_{r,3}(\xi)=-r^2\int_{S_r(r\xi)}\pi(\nu,\nu)^2\,d\mu_\delta.
\]
By the decay of $k$ it holds $ G_{r,3}(\xi)= \mathcal{O}(1)  $ and by $ \mathrm{Sc} \geq -\mathcal{O}(\eta_1|x|^{-4}) $
\[
2 r \int_{\mathbb{R}^3 \setminus B_r(r \xi)} \mathrm{Sc} \, dv_\delta \geq -2 r \int_{\mathbb{R}^3 \setminus B_{r(1 - |\xi|)}(0)} \mathcal{O}(\eta_1|x|^{-4})  \, dv_\delta \geq -\mathcal{O}(\eta_1 |1 - |\xi||^{-1}).
\]
We know that $G_r(\xi)$ is continuous  and that \(G_1\) is strictly convex and blows up as \(|\xi|\nearrow1\).  The bounds for $ G_{r,2}$ and $ G_{r,3}$  imply that
for large \(r\) the continuous function \(G_r\) attains a global minimum on the compact set
\(\{|\xi|\le 1-\Tilde{\delta}\}\) and the minimizer cannot occur at the boundary \( |\xi|=1-\Tilde{\delta}\)
(for \(|\xi|\) near 1 the blow-up of \(G_1\) dominates the small perturbations). Hence for every
large \(r\) there exists at least one interior minimizer \(\xi(r)\) and the  surface \(\Sigma_{\xi(r),r}\) is an area-constrained critical sphere with area \(4\pi r^2\).  Note that this minimum is not necessarily unique, to guarantee uniqueness, we need convexity of $G_r(\xi)$.

Now assuming (\ref{condiuni1}, \ref{condiuni2}), we will show that $\Bar{D}^2 G_r$ is positive definite. By \cite[Lemma 20, Corollary 21]{willcen}  we have that  $ x^i \partial_i (|x|^2 \mathrm{Sc}) \leq \mathcal{O} (\eta_1 |x|^{-2})$ implies that $\Bar{D}^2G_{r,2}$ is positive semidefinite up to an error $ \mathcal{O}(\eta_2 )$, i.e.  $\Bar{D}^2G_{r,2}\geq -c\eta_2  \text{Id}$.

For  $ \Bar{D}^2G_{r,3}$,  like we did for (\ref{taylorexp}) we use again a Taylor expansion at $\xi=0$ obtaining
\begin{equation}
    \begin{split}
    \Bar{D}^2G_{r,3}(a,a) =&   -  r^2\int_{S_{r}(r\xi)}    2\mathcal{D}W_2\left(r \delta( a, \nu ), r \delta( a, \nu ) \right) d\mu_\delta  \\
    =&   -  r^2\int_{S_{r}(0)}   2 \mathcal{D}W_2\left(r \delta( a, \nu ), r \delta( a, \nu ) \right) d\mu_\delta  + \mathcal{O}(|\xi| )\\
    =&\mathcal{O}(\eta_2 ) + \mathcal{O}(|\xi| ),
    \end{split}
\end{equation}
where we used that each term of the second variation depends on the radial directions,  $\pi(\rho,\rho)$ $(\nabla_\rho \pi)(\rho,\rho)$ or $k(\rho,\cdot) $.  Then we have that for $ \eta_i$ and $ |\xi|$ small enough $\Bar{D}^2 G_r$ is positive definite.   Repeating the argument used in Theorem \ref{uniqueness} we obtain the  uniqueness.
\end{proof}
\begin{remark}
$i)$ Recall that the curvature condition 
$\mathrm{Sc} \ge -\mathcal{O}(\eta_1 |x|^{-4})$ is implied by 
$x^i \partial_i (|x|^2 \mathrm{Sc}) \le \mathcal{O}(\eta_1 |x|^{-2}),$ as shown in~\eqref{estilow}. 

$ii)$ Note that for this family of surfaces not only the outer radius diverges but also the inner radius  $\inf_{x \in \Sigma(\lambda) } |x|$ does.

$iii)$ The key reason why, under these weaker assumptions, we cannot guarantee that the family of surfaces constructed above forms a foliation  is the lack of control on the size of the translation parameter~$|\xi|$.

Indeed, revisiting the proof of Lemma~\ref{convexity}, the crucial estimate for the radial derivative 
\(|\xi|^{-1}\xi^i(\partial_i G_{r,3})(\xi)\)
involves, via the expansion~\eqref{estitaylor}, the terms depending on the radial components of~$k$ and~$\pi$, namely
\(\pi(\rho,\rho)\), \((\nabla_\rho \pi)(\rho,\rho)\), and \(k(\rho,\cdot)\).
From this, one obtains
\[
|\xi|^{-1} \xi^i (\partial_i G_{r,3})(\xi)
\ge -\mathcal{O}(\eta_2^2) - \mathcal{O}(\eta_2 |\xi|) - \mathcal{O}(|\xi|^2),
\]
and similarly,
\[
|\xi|^{-1} \xi^i (\partial_i G_{r,2})(\xi)
\ge -\mathcal{O}(\eta_1) - \mathcal{O}(\eta_1 |\xi|) - \mathcal{O}(|\xi|^2).
\]

Combining these inequalities gives a weaker estimate for the total derivative:
\begin{equation}\label{estidgcolo}
    \bar{D}_\xi G_r
    \ge 256\pi |\xi|
    - \mathcal{O}(\eta_1)
    - \mathcal{O}(\eta_2)
    - \mathcal{O}(|\xi|^2).
\end{equation}
Compared with~\eqref{DGexp}, this bound yields less precise control on the growth of~$G_r$ with respect to~$\xi$.
In particular, we can no longer obtain an estimate for~$|\xi|$ as in~\eqref{estixi}. 
This, in turn, results in a weaker bound on 
\((\Bar{D} G'_{3,r})|_{\xi(r)}(a)\) (see~\eqref{dg3esti}), 
which also implies that we do not have a good estimate on~$|\xi'|$, 
a quantity that was crucial in proving that the family of surfaces forms a foliation.

\end{remark}

\subsection{Large-sphere limit}

A natural question is whether the Hawking energy, evaluated along the foliation, approaches the ADM energy in the large-sphere limit. Before we proceed, note that in an asymptotically Schwarzschild manifold of mass $m$, the ADM energy is exactly the parameter 
$m$. This identification might seem confusing at first, so let’s clarify why it holds:

\begin{remark}\label{remarkm}
In a time-symmetric slice of Schwarzschild spacetime, since the   ADM momentum vanishes, the ADM energy and the ADM mass coincide; both equal 
 the parameter $m$ of the metric.  In our more general scenario, we still have $E_{ADM}=m$ and the ADM momentum need not vanish; as a result, the parameter $m$ is not the ADM mass, which, in our case, is $\sqrt{m^2-|p_{ADM}|^2} $.
\end{remark}
Before continuing, we will need the following result from \cite[Proposition A1]{Nerz2}:
\begin{proposition}\label{nertzlemma}
Let $(M, g, x)$ be a $\mathcal{C}^{2, \frac{1}{2}+\epsilon}$-asymptotically flat Riemannian manifold. The ADM energy of $(M,g)$ is defined as
\[
E_{ADM} := \lim_{r \to \infty}  \frac{r}{8\pi} \int_{S_r^2} \left(  \frac{\mathrm{Sc}}{2} - \mathrm{Ric}(\nu,\nu) \right) d\mu ,
\]
provided this limit exists, where $S_r^2 = x^{-1}(S_r^2(0))$, $\nu$ is the outward unit normal, and $d\mu$ is the induced surface measure. Assume $\Sigma \to M \setminus K$ is a closed hypersurface enclosing $K$, i.e., $(M \setminus K)\setminus\Sigma$ consists of two connected subsets of $M$ and $K$ is contained in the relatively compact one. For any constant $c \geq 0$ there is a constant $C = C(\epsilon,  c)$ such that the existence of a vector $z \in \mathbb{R}^3$ with
\[
|z| \leq c r_{min}, \qquad 
\max_{\Sigma} \left| \nu - \frac{x - z}{\, r_{min}} \right| \leq c r_{min}^{-\frac{1+\epsilon}{2}}, \qquad 
|\Sigma| \leq c r_{min}^2,
\]
implies
\[
\left|\, E_{ADM} + \frac{r_{min}}{8\pi} \int_{\Sigma} \bigl(\mathrm{Ric}(\nu,\nu) - \frac{\mathrm{Sc}}{2}\bigr) \, d\mu \right| \leq \frac{C}{r_{min}^\epsilon},
\]
where $r_{min}:= \min_{\Sigma}|x|$, and $\nu$, $\mu$ denote the outer unit normal and the surface measure of $\Sigma \to M$, respectively.
\end{proposition}

The Ricci–flux identity for the ADM mass used here goes back to Ashtekar–Hansen \cite{ashtekar1978unified} and Chruściel \cite{chrusciel1986remark}; see also the modern formulations by Miao–Tam \cite{miao2016evaluation} and Herzlich \cite{herzlich2016computing}.

 We now prove:
\begin{theorem}\label{largesphere}
Let $\{\Sigma_r\}$ denote either the leaves of the foliation from Theorem~\ref{exisfoli} or the on-center family from Theorem~\ref{family}, indexed by their area radius $r$ (so $|\Sigma_r|=4\pi r^{2}$). Then $ \mathcal{E}(\Sigma_r)\to m $ as $r \to \infty$.
\end{theorem}
\begin{proof}
We begin by observing that the estimate from Proposition~\ref{nertzlemma}
may be applied in a \emph{translated asymptotically flat chart} centered at the vector 
$z = r\xi(r)$ associated with the surface $\Sigma(r) $.
Indeed, define the translated coordinates 
\[
x' := x - z,
\qquad 
\hat r := \min_{x\in\Sigma}|x - z|.
\]
Since the asymptotically flat metric satisfies 
$g_{ij}(x) = \delta_{ij} + \mathcal{O}(|x|^{-1})$ and its derivatives decay as 
$\mathcal{O}(|x|^{-2})$, we have for large $|x'|$ that $|x'+z| \simeq |x'|$
whenever $\frac{|z|}{|x'|} \text{ remains bounded.}
$

In our case $|z| = r|\xi|$ and $|\xi|\le 1-\tilde{\delta}$, hence 
$|z|\le (1-\tilde{\delta})r \le C\,\hat r$, and therefore 
$\frac{|z|}{|x'|}$ is uniformly bounded in the asymptotic region.
Consequently, the metric remains asymptotically flat 
in the translated chart:
\[
g_{ij}(x'+z) = \delta_{ij} + \mathcal{O}(|x'|^{-1}), 
\qquad 
\partial g_{ij}(x'+z) = \mathcal{O}(|x'|^{-2}),
\]
and the same decay holds for higher derivatives.
Thus, Proposition~\ref{nertzlemma} applies  with $\hat r$ in place of the 
original inner radius $r_{min}$.
This translation merely recenters the coordinate system near the geometric 
center $z=r\xi$ of $\Sigma(r)$ and does not affect any asymptotic 
quantities such as the ADM energy.

We now check that the remaining hypotheses Proposition \ref{nertzlemma} are satisfied for the surface $\Sigma(r)$ in the 
translated chart, with $z=r\xi$ and $\hat r = \min_{\Sigma(r)}|x-z|$.  First note that  $\Sigma(r)$ is a normal graph of $u_r$ over $S_r(r\xi(r))$ and  by Proposition \ref{17EK}, $\|u_r\|_{\mathcal{C}^2}$ is bounded independent of $r$ then  
$$\hat r = \min_{x\in\Sigma(r)}|x-r\xi(r)|\leq \|u_r\|_{\mathcal{C}^0} +r\leq C r$$ 
for some constant $C$.  Given $x\in \Sigma(r)$ then there exit a $y \in S_r(r\xi(r)) $ such that $x=y+u_r(y)\frac{y-r\xi(r)}{r}$ and 
$$r=|y-r\xi(r) |=|x-u_r(y)\frac{y-r\xi(r)}{r}-r\xi(r) | \leq  |x-r\xi(r) |+\|u_r\|_{\mathcal{C}^0} \quad \forall x\in \Sigma(r) .$$ 
Then $r \leq \hat{r} + \|u_r\|_{\mathcal{C}^0}$ and  there exist a constant $C'$ such that $r \leq  C'\hat{r}$. This implies that $|r -\hat{r}| \leq \|u_r\|_{\mathcal{C}^0} $ and there is a constant $\tilde{C}$ such that $ \frac{1}{\tilde{C}}\hat{r}\leq r \leq\tilde{C} \hat{r}$,  therefore we can consider the two radius to be equivalent for our purposes.  

Note that in these new coordinates the vector $z$ of the proposition is zero, then it follows directly that $|z|=0 \leq \hat{r}$ and
$|\Sigma(r)| \leq c \hat{r}^2$. It remains to check the condition  on the normal vector.  Note that since $\Sigma(r)$ is a graph over $S_r(r\xi(r))$, it holds that  $\left| \nu - \frac{x - r\xi(r)}{ r} \right|\leq C r^{-1}\|u_r\|_{\mathcal{C}^1} $ then with this we can see that.
\begin{equation*}
    \begin{split}
  \left| \nu - \frac{x' }{\, \hat{r}} \right| &=  \left| \nu - \frac{x - r\xi(r)}{ \hat{r}} \right| \leq \left| \nu - \frac{x - r\xi(r)}{ r} \right| +\left|\frac{x - r\xi(r)}{ r} - \frac{x - r\xi(r)}{ \hat{r}} \right|\\
  &\leq Cr^{-1} \|u_r\|_{\mathcal{C}^1} + \left| \frac{1}{\hat{r}}-\frac{1}{r} \right| |x - r\xi(r)|
  \leq   C r^{-1}\|u_r\|_{\mathcal{C}^1}+ C r^{-2}|x - r\xi(r)|  \\
  &\leq C\hat{r}^{-1}
    \end{split}
\end{equation*}
Hence the assumptions of Proposition \ref{nertzlemma} hold for the translated 
surfaces $\Sigma(r)$, and we may apply its conclusion with $\hat{r}$ in 
place of $r_{min}$.

 By  Proposition \ref{nertzlemma} and that $|r -\hat{r}| \leq \|u_r\|_{\mathcal{C}^0} $  we have 
\begin{equation}
    \begin{split}
       m&= \frac{\hat{r}}{16\pi} \int_{\Sigma(r)} (\mathrm{Sc} - 2\,\mathrm{Ric}(\nu,\nu)) \, d\mu + \mathcal{O}(r ^{-\epsilon})\\
       &=  \frac{r}{16\pi} \int_{\Sigma(r)} (\mathrm{Sc} - 2\,\mathrm{Ric}(\nu,\nu)) \, d\mu+ \frac{\hat{r} -r  }{16\pi} \int_{\Sigma(r)} (\mathrm{Sc} - 2\,\mathrm{Ric}(\nu,\nu)) \, d\mu + \mathcal{O}(r ^{-\epsilon}) \\
       &= \frac{r}{16\pi} \int_{\Sigma(r)} (\mathrm{Sc} - 2\,\mathrm{Ric}(\nu,\nu)) \, d\mu + \mathcal{O}(r ^{-\epsilon})
    \end{split}
\end{equation}

 Then by the  Gauss equation, Gauss-Bonnet theorem and  that by the decay of $k$ it holds $\sqrt{\frac{|\Sigma(r)|}{16\pi}}  \int_{\Sigma(r)} P^2\,d\mu= \mathcal{O}(r ^{-\epsilon})$,  we have 
\begin{equation}
    \begin{split}
        |m - \mathcal{E}(\Sigma)| 
\leq&   \left| m - \frac{r}{16\pi} \int_{\Sigma(r)} (\mathrm{Sc} - 2\,\mathrm{Ric}(\nu,\nu)) \, d\mu \right|
+ \left| \frac{r}{16\pi} \int_{\Sigma(r)} (\mathrm{Sc} - 2\,\mathrm{Ric}(\nu,\nu)) \, d\mu - \mathcal{E}(\Sigma(r)) \right|\\
\leq& \left|  \frac{r}{16\pi}\int_{\Sigma(r)}\left(\mathrm{Sc}^{\Sigma} - \frac{H^2}{2} + |\mathring{B}|^2\right)d\mu 
- \sqrt{\frac{|\Sigma(r)|}{16\pi}}\left(1 - \frac{1}{16\pi} \int_{\Sigma(r)}H^2  d\mu\right) \right| + \mathcal{O}(r ^{-\epsilon}) \\
\leq & \frac{r}{16\pi}\int_{\Sigma(r)} |\mathring{B}|^2d\mu  + \mathcal{O}(r ^{-\epsilon})
    \end{split}
\end{equation}
for some $\epsilon>0$.

To estimate $\int_{\Sigma(r)} |\mathring{B}|^2d\mu$ first note that since $\Sigma(r)$ is a Hawking surface it satisfies equation (\ref{eulag}), which like in (\ref{shorteug}) and (\ref{explaineulag}) can be written as $ H\lambda+ W_1 + W_2=0$,and  by the decay of $k$ it reduces to $ H\lambda+ W_1 =\mathcal{O}(r^{-5})$. Dividing the  equation by $H>0$, integrating over $\Sigma(r)$,
and integrating by parts the term $\Delta_\Sigma H/H$ gives
\begin{equation}
    \int_{\Sigma_r}\lambda+ |\nabla^\Sigma  \log H|^2 + \frac{1}{2} |\mathring{B}|^2 + \mathrm{Ric}(\nu,\nu) d\mu = \mathcal{O}(r^{-2}).
\end{equation}
 Since $\lambda$ is nonnegative (see Lemma \ref{expansquafoli}) and by the decay  $\mathrm{Ric}= \mathcal{O}(r^{-3}) $ we have 
\begin{equation}
    \int_{\Sigma(r)} |\mathring{B}|^2d\mu= \mathcal{O}(r^{-1})
\end{equation}
However this decay is not good enough to prove our result, in orther to improve the decacy we will use the  the \emph{bootstrapping result}
of Cederbaum--Sakovich \cite[Proposition 4.5]{STCMC},
which is an adaptation of the estimates in
Nerz \cite[Proposition 2.4]{Nerz} to the
spacetime mean--curvature (STCMC) setting. Roughly speaking, their result asserts that for a family of closed surfaces
$\Sigma\subset M$ which are small normal graphs over coordinate spheres
in an  $\mathcal{C}^{2, \frac{1}{2}+\varepsilon}$-asymptotically flat manifold, if the mean curvature satisfies
\[
H = \frac{2}{r} + \mathcal{O}(r^{-1-\varepsilon}),
\]
and if $\|\mathring{B}\|_{L^{2}(\Sigma)}$ is sufficiently small,
then one obtains the quantitative decay
\[
\|\mathring{B}\|_{L^{2}(\Sigma)} = \mathcal{O}(r^{-1/2-\varepsilon}),
\qquad
\|\mathring{B}\|_{L^{\infty}(\Sigma)} = \mathcal{O}(r^{-3/2-\varepsilon}),
\]
In our case, the spheres $\Sigma(r)$ constructed in
Theorem~\ref{exisfoli} and Theorem~\ref{family}
satisfy all these hypotheses:
they are small normal graphs over $S_{r}(r\xi(r))$,
the ambient metric $(M,g)$ is $\mathcal{C}^{2,1/2+\varepsilon}$--asymptotically flat, by the analysis above
$H=\tfrac{2}{r}+\mathcal{O}(r^{-1-\varepsilon})$
and $\|\mathring{B}\|_{L^{2}(\Sigma(r))}=\mathcal{O}(r^{-\frac{1}{2}})$.
Hence we can apply
\cite[Proposition 4.5]{STCMC}
directly to conclude that
\[
\|\mathring{B}\|_{L^{2}(\Sigma(r))}=\mathcal{O}(r^{-1/2-\varepsilon}).
\]
Substituting this decay into
\[
|m-\mathcal{E}(\Sigma(r))|
\le \frac{r}{16\pi}\!\int_{\Sigma(r)}|\mathring{B}|^{2}\,d\mu
+\mathcal{O}(r^{-\epsilon})
\]
gives the desired convergence
$\mathcal{E}(\Sigma_{r})\to m$ as $r\to\infty$.
\end{proof}
With this result, we have established that the Hawking energy evaluated on Hawking surfaces satisfies the large-sphere limit. A study of the small-sphere limit on such surfaces can be found in \cite{Alex, diaz2023local}. 
\begin{remark}
    Note that this large-sphere limit converges to the ADM energy and not the ADM mass $\sqrt{E_{ADM}^2- p_{ADM}^2}$, the reason for this is that when considering asymptotically flat manifolds (under standard decay assumptions), the decay of $k$ is too rapid to capture the momentum contribution in the limit. This not only happens with this foliation, the same occurs for the STCMC foliation of Cederbaum and Sakovich, see \cite[Proposition 5.6]{STCMC}. For this reason, we consistently view $\mathcal{E}(\Sigma)$  here as a quasi‐local energy, not a mass.
\end{remark}
\section{Center of the foliation}\label{sectioncenter}

The study of the center of geometric foliations in asymptotically flat manifolds has attracted substantial attention. The first major result in this direction was due to Huisken and Yau \cite{HY}. In an asymptotically flat, time-symmetric initial data set they constructed a unique foliation of the end by constant mean curvature (CMC) spheres. Under additional asymptotic symmetry assumptions,  the coordinate center of this foliation coincides with the Hamiltonian center of mass \(\vec{C}_{\mathrm{H}} = (C^1_{\mathrm{H}}, C^2_{\mathrm{H}}, C^3_{\mathrm{H}})\), which is given by
\begin{equation}\label{eq:BOM_center}
    C^l_{\text{H}} :=
    \frac{1}{16\pi E} \lim_{r \to \infty} \int_{|\vec{x}|=r}
    \left[
    x^l \sum_{i,j} \left( \partial_i g_{ij} - \partial_j g_{ii} \right) \frac{x^j}{r}
    - \sum_i
    \left(
    g_{il} \frac{x^i}{r} - g_{ii} \frac{x^l}{r}
    \right)
    \right]
    d\mu_\delta,   
\end{equation}
where \(E\) is the ADM energy and \(d\mu_\delta\) is the Euclidean volume element on the coordinate sphere of radius \(r\). See \cite{beigmurch, regge1974role} for further discussion of this formulation and its invariance properties.

This construction was generalized to the dynamical setting by Cederbaum and Sakovich in \cite{STCMC}, who constructed a unique foliation by STCMC (spacetime constant mean curvature) surfaces in appropriate asymptotically flat initial data sets. They proved that the coordinate center of that foliation converges, and identified its limit as a suitably generalized center of mass—the STCMC center of mass—which reduces to the Huisken–Yau/Hamiltonian center in the time-symmetric case. The center of the STCMC foliation can be written as the Hamiltonian center of mass plus the \(k\)-dependent correction
\begin{equation}\label{correctstcmc}
    \lim_{r\to\infty}\frac{r^2}{32\pi E_{ADM}}\int_{S_r(0)} \big(\pi(\nu,\nu)\big)^2\, \nu \, d\mu_\delta.
\end{equation}
Eichmair and Koerber \cite{willcen} showed that, in general, the center of the Willmore foliation need not converge. However, under more restrictive assumptions on the scalar curvature, the center of the Willmore foliation is well defined and coincides with the Hamiltonian center of mass.
\begin{theorem}[{\cite[Theorem 2]{willcen}}]\label{willcenmain}
  Let \((M, g)\) be \(\mathcal{C}^4\)-asymptotic to Schwarzschild with mass \(m > 0\) and Hamiltonian
center of mass \(C_H = (C_H^1, C_H^2, C_H^3)\) and suppose that the scalar curvature satisfies, as \(x \to \infty\),
\begin{equation}
  \tilde{x}^i \partial_i(|\tilde{x}|^2 \mathrm{Sc} (\tilde{x})) \leq o(|x|^{-3}), \quad    \mathrm{Sc}(\tilde{x}) - \mathrm{Sc}(-\tilde{x}) = o(|x|^{-5}),
\end{equation}
where \(\tilde{x} = x - C_H\). Then the center of the foliation by Willmore spheres \(C_{ACW}\) exists and \(C_H = C_{\text{ACW}}\).  
\end{theorem} 
It is then natural to ask whether the center of the foliation by Hawking surfaces  is well defined and if it is somehow related to the STCMC center of mass constructed in \cite{STCMC}. To address this, we study the coordinate center of the foliation produced in Theorem \ref{exisfoli}.

The analysis in this section follows the general strategy developed by 
Eichmair and Koerber \cite{willcen} for the Willmore foliation, adapted here to the setting of area–constrained Hawking spheres.
In particular, several steps---such as the quantitative estimates for the center parameter and the control of geometric quantities along the family of surfaces---parallel
the corresponding arguments in~\cite{willcen}, 
with the necessary modifications to account for the additional $k$--terms 
and the different Euler--Lagrange structure of the Hawking functional. We therefore outline only the main differences and highlight the new estimates specific to the present setting. First we outline the main steps to guide the reader.

$i)$  \textbf{Improved control of the translation parameter.}
  Under strengthened parity/decay hypotheses  we sharpen the bound on the centering parameter \(\xi(r)\) so that it cannot drift, this prevents the leaves and their centers from escaping to infinity.

   $ii)$ \textbf{Recentering and removing the constant mode.} We rewrite our surfaces as graphs over coordinate spheres with a radius $r-2$. This radius is chosen so that the graph function \(u\) has vanishing average; this eliminates the constant mode and isolates the genuine translation information in \(\xi(r)\).

$iii)$  \textbf{Schwarzschild comparison for geometric quantities.}   geometric data of \(\Sigma(r)\) is  compared to those of the corresponding coordinate sphere in the Schwarzschild metric. Here most of the $k$ terms can be treated as a perturbation therefore required estimates coincide with the comparison bounds in~\cite{willcen} and carry over  in our setting.

$iv)$ \textbf{Expansion for the center and identification of the limit.} Finally, again following \cite{willcen} we get a nice expression for $r\xi(r)$ which allow us to study the center of the foliation.

As observed earlier, the foliations need not be centered in general. This indicates that centering requires stronger asymptotic symmetry and decay. Accordingly, we impose the following additional hypotheses on the scalar curvature and the tensor $k$:  
\begin{equation}\label{paritycond}
   |\mathrm{Sc}^{\text{odd}} |=\mathcal{O} ( |x|^{-5}), \quad  |k^{\text{even}}| + |x||(\nabla k)^{\text{odd}}| = \mathcal{O} (|x|^{-3})
\end{equation}
and 
\begin{equation}
    x^i \partial_i (|x|^2 \mathrm{Sc}) \leq \mathcal{O}( |x|^{-3}),
\end{equation}
 These symmetry assumptions are a direct consequence of the Regge-Teitelboim conditions introduced in \cite{regge1974role}. In particular, they guarantee the convergence of the Hamiltonian center of mass.  Note also that under \eqref{paritycond} the correction term (\ref{correctstcmc}) of the STCMC center of mass vanishes.

One of the main reasons to have the extra assumptions is to improve the decay of $\xi(r)$ in the foliation.
\begin{lemma}\label{extixi}
   Under the assumptions of before, there holds, as \(r \to \infty\),
\[
\xi(r) = \mathcal{O} (r^{-1}).
\]
\end{lemma}
\begin{proof}
As in Lemma \ref{convexity} we have 
\begin{equation}
    \Bar{D}_\xi G_r = \Bar{D}_\xi G_1 + \Bar{D}_\xi G_{r,2}+\Bar{D}_\xi G_{r,3}, 
\end{equation}
 using $ x^i \partial_i (|x|^2 R) \leq \mathcal{O}( |x|^{-3})$, the parity conditions and using the same argument as in \cite[Lemma 11]{willcen} (or \cite[Lemma 24]{eichko}) we have 
\begin{equation}
   -2 r^2 \int_{S_{\xi,r}} |\xi|^{-1} \delta( \xi, \nu ) \,\mathrm{Sc} \, d\mu_{\delta} \geq  - \mathcal{O} (r^{-1}). 
\end{equation}
Now to estimate $\Bar{D}_\xi G_{r,3} $ we use again the Taylor expansion (\ref{taylorexp}) 
\begin{equation}
\begin{split}
  |\xi|^{-1} \xi^i (\partial_i G_{r,3})(\xi) =& -r^2 \int_{S_{r}(0)}   2 W_2 \alpha d\mu_{\delta} -  r^2\int_{S_{r}(0)}   2 \mathcal{D}W_2\left(r \delta( \frac{\xi}{|\xi|}, \nu ), r \delta( \xi, \nu ) \right) d\mu_{\delta} \\
 &- \frac{r^2}{2} \Bar{D}^2 \left(\int_{S_{r}(\xi r)}  2 W_2 \alpha d\mu_{\delta}\right)_{|\xi = \Tilde{\xi}}(r \xi, r \xi).\\
 \end{split}
\end{equation}
This time, using the parity conditions on $k$ and $\nabla k$ it is direct to see that the first term satisfies 
\begin{equation*}
\begin{split}
 r^2 \int_{S_{r}(0)}   2 W_2 \alpha d\mu_{\delta}   &= 2r^2 \int_{S_r(0)}  \pi(\nu, \nu) r \delta( \frac{\xi}{|\xi|}, \nu ) \nabla_{\nu}\pi(\nu,\nu) +2 \pi(\nu, \nu)    k( \frac{\xi}{|\xi|} , \nu )+  \delta( \frac{\xi}{|\xi|}, \nu )  \pi(\nu, \nu)^2  d\mu_{\delta}\\
 &=\mathcal{O} (r^{-1}).
 \end{split}
\end{equation*}
 As in (\ref{estitaylor}) we can estimate the other two terms by $O( \eta_3 |\xi|)  $, obtaining 
 \begin{equation}
\begin{split}
    |\xi|^{-1} \xi^i (\partial_i G_{r,3})(\xi) \geq& -\mathcal{O}(r^{-1}) -\mathcal{O}(\eta_3 |\xi|).
   \end{split}
\end{equation}
 Using that \((DG_{r})(\xi(r)) = 0\), and
\[
|\xi(r)|^{-1}  \xi(r)^i (\partial_i G_1)(\xi(r)) \geq 256 \pi |\xi(r)|.
\]
 we find   
 \begin{equation}
     0 \geq 256 \pi |\xi(r)| - O( \eta_3  |\xi(r)|) - \mathcal{O} (r^{-1}) 
 \end{equation}
and then we obtain the result for $\eta_3$ small enough.
\end{proof}
This estimate is crucial to identify a consistent geometric center for the foliation; without it, the notion of center could drift with the radius.

We have seen that each leaf $\Sigma_{\xi, r}$ of the foliation can be expressed as the graph over $ S_{\xi, r}$ of a function $u= -2 + \mathcal{O}(|\xi|^2)$, this implies that when comparing $\Sigma_{\xi, r}$ with $S_{\xi,r}$ the difference would be given by $-2 + \mathcal{O}(|\xi|^2) $. The comparison between the two surfaces will be important when trying to study the center of the surfaces $\Sigma_{\xi, r}$ in terms of the center of the spheres $ S_{\xi,r}$, that is why it is more practical to consider the spheres $S_{\xi,r}$ with a different radius $\tilde{r} = r - 2 $ and have the difference between $\Sigma_{\xi, r}$ and $ S_{\xi,r}$ be given by the error term $\mathcal{O}(|\xi|^2) $. Summarizing, following \cite{willcen} we define:
\[
\tilde{u}_{\xi,r} = u_{\xi,r} + 2
\]
so that
\[
\Sigma_{\xi,r} = \Sigma_{\tilde{\xi},\tilde{r}}(\tilde{u}_{\xi,r})
\]
with
\[
\tilde{r} = r - 2 \quad \text{and} \quad \tilde{\xi} = (r - 2)^{-1} r \xi.
\]
Where $r\xi = \tilde{r} \tilde{\xi} $, therefore $S_{\tilde{\xi},\tilde{r}} $ and $S_{\xi,r} $ have the same center.

We abbreviate \(u_{\xi,r}\) and \(\tilde{u}_{\xi,r}\),   by \(u\) and \(\tilde{u}\) respectively. Moreover, we let \(\Lambda_0(S_{\tilde{\xi},\tilde{r}}) \subset \mathcal{C}^\infty(S_{\tilde{\xi},\tilde{r}})\) be the space of constant functions and \(\Lambda_0^\perp(S_{\tilde{\xi},\tilde{r}})\) be its orthogonal complement. We abbreviate \(\Lambda_0(S_{\tilde{\xi},\tilde{r}})\) by \(\Lambda_0\) and \(\Lambda_0^\perp(S_{\tilde{\xi},\tilde{r}})\) by \(\Lambda_0^\perp\).

 With these definitions and the previous results, we obtain the following result, which is the same as  \cite[Lemma 13]{willcen} and has the same proof. 
\begin{lemma}\label{decayutil}
    There exists \(\Tilde{\delta} \in (0, 1/4)\) such that
\[
\tilde{u} = \mathcal{O} (|\xi|^2) +  \mathcal{O}(r^{-1})
\]
and, uniformly for every \(\xi \in \mathbb{R}^3\) with \(|\xi| < \Tilde{\delta}\) as \(r \to \infty\),
\[
\operatorname{proj}_{\Lambda_0} \tilde{u} = -r^{-1} - \frac{1}{16\pi} r^{-1} \int_{S_{r, \xi}} [\operatorname{tr} \sigma - \sigma(\nu, \nu)] \, d\mu_{\delta} + \mathcal{O} (r^{-2}) + \mathcal{O}(r^{-1} |\xi|^2),
\]
where $\sigma $ is the decay tensor of the metric introduced in the Definition \ref{asymptoticsc} of an asymptotically Schwarzschild manifold. These expansions can be differentiated once with respect to \(\xi\).
\end{lemma}
As mentioned above, this lemma shows that as $\Sigma_{\xi,r}$ is the graph of $\tilde{u}_{\xi,r}$ over $S_{\tilde{\xi},\tilde{r}} $ the difference between  $\Sigma_{\xi,r}$   and $S_{\tilde{\xi},\tilde{r}} $ is of order $\mathcal{O}(r^{-1})$. This is why we subtract the constant part of $u $ to define $\tilde{u}$. Note also that in the proof of the lemma, $k$ does not play any role.

We also have the following result equivalent to \cite[Lemma 14]{willcen}
\begin{lemma}
    There exists \(\Tilde{\delta} \in (0, 1/4)\) such that, uniformly for every \(\xi \in \mathbb{R}^3\) with \(|\xi| < \Tilde{\delta}\) as \(r \to \infty\),
\begin{equation}\label{intemean}
\int_{\Sigma_{\xi,r}} H^2 \, d\mu = \int_{S_{\tilde{\xi},\tilde{r}}} H^2 \, d\mu - 64 \pi r^{-3} - 4 r^{-3} \int_{S_{\xi,r}} \operatorname{tr} \sigma - \sigma(\nu, \nu)\,  d\mu_{\delta} + \mathcal{O}(r^{-3} |\xi|^2) + \mathcal{O}(r^{-2} |\xi|^4) + \mathcal{O}(r^{-4})
\end{equation}
\begin{equation}\label{comparp}
    \int_{\Sigma_{\xi,r}} P^2 \, d\mu = \int_{S_{\tilde{\xi},\tilde{r}}} P^2 \, d\mu_{\delta} + \mathcal{O}(r^{-4}).
\end{equation}
This expansion may be differentiated once with respect to \(\xi\).
\end{lemma}
Note that (\ref{intemean}) is the same expression obtained in \cite[Lemma 14]{willcen} ($k$ doesn't play any role) and (\ref{comparp}) is obtained directly by using the decay of $k$. This shows that we can reuse many of the results of \cite{willcen} concerning the expression (\ref{intemean}).  We can then apply directly \cite[Proposition 18]{willcen} to obtain 
\begin{equation}\label{precenter}
    \begin{split}
       256 \pi r \xi(r) =& 2r^3 \int_{S_r(r \xi(r))} \mathrm{Sc} \, \nu\, d\mu_{\delta}-8 r \int_{S_r(0)}  \Bar{D} \operatorname{tr} \sigma -  \Bar{D}\sigma(\nu, \nu) -\frac{2}{r} \operatorname{tr} \sigma\, \nu \,d\mu_{\delta} \\
       &+ r^3\Bar{D} \int_{S_{\tilde{\xi},\tilde{r}}} P^2 \, d\mu_{\delta} + \mathcal{O}(r^{-1}). 
    \end{split}
\end{equation}
It was shown in \cite[Proposition 18]{willcen} that the second term of the previous expression is related to the Hamiltonian center of mass. To see this, note that the center of mass of a Schwarzschild manifold is given by the zero vector, now since our manifold is asymptotic to Schwarzschild, the center of mass will only depend on the perturbation term of the metric (the tensor $\sigma$).  Using the definition of the Hamiltonian center of mas  (\ref{eq:BOM_center}) and its linearity we have 
\[
z^\alpha := \frac{1}{32\pi\,r}
\int_{S_r(0)} 
\sum_{i,j=1}^3
\Big(
x^\alpha x^j \big[ (\partial_i \sigma)(e_i, e_j) - (\partial_j \sigma)(e_i, e_i) \big]
-
x^i \big[ \sigma(e_i, e^\alpha) - \delta_i^{\ \alpha}\,\sigma(e_i,e_i) \big]
\Big)
\, \mu_\delta,
\]
for $\alpha = 1,2,3$, and $\lim_{r \to \infty} z^\alpha = C_H^\alpha .$ Using integration by parts and the decomposition
\(e_\alpha = e_\alpha^{\perp} + e_\alpha^{\top}\)
with respect to \(\delta\),
we obtain
\begin{equation}
z^\alpha
=
\frac{1}{32\pi\,r}
\int_{S_r(0)}
\big(
(\partial_\alpha \sigma)(\nu, \nu)
- \partial_\alpha \tr \sigma
+\frac{2}{r} \, \tr\sigma \,\nu^\alpha
\big)
\, d\mu_\delta. 
\end{equation}
Then replacing this second term and the variation  $\Bar{D} \int_{S_{\tilde{\xi},\tilde{r}}} P^2 \, d\mu_{\delta}= \int_{S_{\tilde{\xi},\tilde{r}}}2 r W_2 \nu \, d\mu_{\delta}$ into (\ref{precenter}), we arrive at the following result which is  the dynamical version of  \cite[Corollary 19]{willcen}.
\begin{proposition}
Under the assumptions of before,    there holds, as \(r \to \infty\),
\begin{equation}\label{lambdaxi}
r \xi(r) = C_H + \frac{1}{128 \pi} r^3 \int_{S_r(r \xi(r))} \mathrm{Sc} \, \nu +r W_2 \nu \, d\mu_{\delta} + \mathcal{O}(r^{-1}).
\end{equation}
\end{proposition}
Now since the surfaces $ \Sigma_{\xi(r), r}$ are graphs of $\tilde{u}= \mathcal{O}(r^{-1}) $ over  $S_{\tilde{r}}(\tilde{r} \tilde{\xi}(r))$ and $ \tilde{r} \tilde{\xi}=r \xi $ then it holds
\begin{equation}   
|\Sigma_{\xi(r), r}|^{-1} \int_{\Sigma_{\xi(r), r}} x^\ell \, d\mu =|S_{\tilde{r}}(\tilde{r} \tilde{\xi}(r))|^{-1} \int_{S_{\tilde{r}}(\tilde{r} \tilde{\xi}(r))} x^\ell \, d\mu +\mathcal{O}(r^{-1})= r \xi(r)^\ell + \mathcal{O}(r^{-1}).
\end{equation}
Then, putting everything together, we obtain the following result. 
\begin{theorem}\label{centerfolia}
   Let $(M,g,k)$ be $ \mathcal{C}^4$-asymptotic to Schwarzschild with mass $m>0$ satisfying the conditions of Theorem \ref{exisfoli}.  Furthermore, assume that 
   \begin{equation}
   |\mathrm{Sc}^{\text{odd}} |=\mathcal{O} ( |x|^{-5}), \quad  |k^{\text{even}}| + |x||(\nabla k)^{\text{odd}}| = \mathcal{O} (|x|^{-3})
\end{equation}
and 
\begin{equation}
    x^i \partial_i (|x|^2 \mathrm{Sc}) \leq \mathcal{O}( |x|^{-3}).
\end{equation}
 Then there exists an on-center foliation by Hawking surfaces, and its center is given by 
\begin{equation}\label{centerhawking}
    C_{f}= C_H+ \lim_{r \to \infty}\frac{1}{128 \pi} r^3 \int_{S_r(r \xi(r))}  \mathrm{Sc}\, \nu +r W_2 \nu \, d\mu_{\delta},
\end{equation}
 provided the limit converges, where $W_2= P( \nabla_\nu \tr k - \nabla_\nu k(\nu,\nu )) - 2 \diver_\Sigma (P k(\cdot, \nu))  +\frac{1}{2}H P^2 $ and $C_H$ is the Hamiltonian center of mass.
\end{theorem}
\begin{remark}
Under the stronger decay and parity hypotheses (and in the time-symmetric case), Theorem~\ref{willcenmain} shows that the center of the Willmore foliation agrees with the Hamiltonian center of mass. In general, the Willmore center is sensitive to the asymptotic \emph{distribution} of the scalar curvature and may differ from the Hamiltonian center unless additional strong asymptotic symmetry is imposed; this is evident from \eqref{lambdaxi} with \(k=0\) and was also observed by Eichmair–Koerber \cite{willcen}. 

In the genuinely dynamical case, an analogous phenomenon occurs with the second fundamental form \(k\): the center defined in \eqref{centerhawking} will \emph{not} in general coincide with the STCMC center. Although the \(k\)-dependent correction in \(C_{\mathrm{STCMC}}\) \eqref{correctstcmc} vanishes under our assumptions, our center retains the \(k\)-dependent term
\[
\lim_{r\to\infty}\frac{r^{4}}{128\pi}\int_{S_r(r\,\xi(r))} W_2\, \nu \, d\mu_\delta,
\]
which need not vanish. Hence \(C_f \neq C_{\mathrm{STCMC}}\) except in special symmetric situations. In short: just as the Willmore center is more sensitive to the scalar curvature than the CMC center, the Hawking foliation is more sensitive to the asymptotic distributions of both \(\mathrm{Sc}\) and \(k\) than the STCMC foliation.
\end{remark}

Going back to (\ref{lambdaxi}), note that we can express the integral of the expression as centered integrals by using a Taylor expansion,  first note that the expression can be written as 
$$r \delta(  \xi(r), e_i ) = \delta(  C_H, e_i ) + \frac{1}{128 \pi} r^3 \int_{S_r(r \xi(r))} \mathrm{Sc} \,\delta( \nu, e_i )  +r W_2 \delta(  \nu,  e_i )  \, d\mu_{\delta} + \mathcal{O}(r^{-1}). $$
for $i=1,2,3$ and where $e_i$ are  coordinate vectors. Here the challenge is how to deal with the integrals since they are centered at $r \xi$. One direct way would be to consider a Taylor expansion at a certain point $q$. We also assume that the parity conditions are centered at $q$.

Now let's see the second term which is equivalent to the  expansion done in (\ref{taylorexp}) and where we denote   $ \alpha=  \delta( e_i, \nu )$ 
\begin{equation*}
    \begin{split}
        \frac{r^4}{128 \pi}  \int_{S_r(r \xi(r))} W_2 \delta( e_i, \nu ) \, d\mu_{\delta}&=    \frac{r^4}{128 \pi}  \int_{S_r(q)}  W_2 \delta( e_i, \nu )  \, d\mu_{\delta}  \\
        &+\frac{r^4}{128 \pi}  \int_{S_{r}(q)}   2 \mathcal{D}W_2\left( \delta( e_i, \nu ),  \delta( r\xi-q, \nu ) \right) d\mu_{\delta} \\
 & + \frac{r^4}{ 256 \pi}  \Bar{D}^2 \left(\int_{S_{r}(\xi r)}  2 W_2 \delta( e_i, \nu ) d\mu_{\delta}\right)_{|\xi = \Tilde{\xi}}(r \xi-q, r \xi-q)
    \end{split}
\end{equation*}
Where $\Tilde{\xi}$ is  some intermediate point on the straight‐line segment from  $|q|$ to $  \xi r $. Now as in Lemma \ref{extixi} we see that the  third term decays as $\mathcal{O}(|\xi|)= \mathcal{O}(r^{-1})$ obtaining 
\begin{equation*}
\begin{split}
     \frac{r^4}{128 \pi}  \int_{S_r(r \xi(r))} W_2 \delta( e_i, \nu ) \, d\mu_{\delta} =&   \frac{r^4}{128 \pi}  \int_{S_r(q)}  W_2 \delta( e_i, \nu )  \, d\mu_{\delta} \\ &+\frac{r^4}{128 \pi}  \int_{S_{r}(q)}   2 \mathcal{D}W_2\left( \delta( e_i, \nu ),  \delta( r\xi-q, \nu ) \right) d\mu_{\delta} +\mathcal{O}(r^{-1})
     \end{split}
\end{equation*}
also note that by the parity conditions on $k$, this expression is bounded.  Now, using again a Taylor expansion for the center of the integrals, we have 
\begin{equation}
    \begin{split}
        &\frac{r^3}{128 \pi}  \int_{S_r(r \xi(r))} \mathrm{Sc} \, \delta( e_i, \nu )   \, d\mu_{\delta} =\frac{r^3}{128 \pi}  \int_{S_r(q)} \mathrm{Sc}\, \delta( e_i, \nu )   \, d\mu_{\delta} \\
        &\quad+ \frac{r^3}{128 \pi}  \int_{S_r(q)}  \Big(  \delta( e_i, \nu )  \delta(  r\xi-q, \nu ) \Bar{D}_{\nu} \mathrm{Sc}+ 3 \delta( e_i, \nu ) \delta(  r\xi-q, \nu ) \frac{\mathrm{Sc}}{r}  - \frac{\mathrm{Sc}}{r} \, \delta( e_i,r\xi-q ) \Big) d\mu_{\delta}\\
        &\quad + \frac{r^3}{256 \pi}  \Bar{D}^2 \left(\int_{S_{r}(\xi r)}  \mathrm{Sc}\,  \delta( e_i, \nu )   \, d\mu_{\delta}\right)_{|\xi = \Tilde{\xi}}(r \xi-q, r \xi-q)\\
        &=    \frac{r^3}{128 \pi}  \int_{S_r(q)}  \Big(  \delta( e_i, \nu )  \delta(  r\xi-q, \nu ) \Bar{D}_{\nu} \mathrm{Sc}+ 3 \delta( e_i, \nu ) \delta(  r\xi-q, \nu ) \frac{\mathrm{Sc}}{r}  - \frac{\mathrm{Sc}}{r} \, \delta( e_i,r\xi-q ) \Big)  d\mu_{\delta}\\
        &\quad + \frac{r^3}{128 \pi}  \int_{S_r(q)} \mathrm{Sc}\, \delta( e_i, \nu )   \, d\mu_{\delta} + \mathcal{O}(r^{-1})
    \end{split}
\end{equation}
Note that by the parity conditions (which we assume centered at $q$) on $\mathrm{Sc}$, this term is bounded. 
Putting everything together, we have the following expression for the center. 
\begin{equation}
    \begin{split}
        r \delta(  \xi(r&), e_i ) =  \delta(  C_H, e_i )+ \frac{r^3}{128 \pi}  \int_{S_r(q)} \mathrm{Sc} \, \delta( e_i, \nu )   \, d\mu_{\delta} \\
        &+\frac{r^3}{128 \pi}  \int_{S_r(q)}  \Big(  \delta( e_i, \nu )  \delta(  r\xi-q, \nu ) \Bar{D}_{\nu} \mathrm{Sc}+ 3 \delta( e_i, \nu ) \delta(  r\xi-q, \nu ) \frac{\mathrm{Sc}}{r}  - \frac{\mathrm{Sc}}{r} \, \delta( e_i,r\xi-q ) \Big)  d\mu_{\delta} \\
       &+ \frac{r^4}{128 \pi}  \int_{S_r(q)}  W_2 \delta( e_i, \nu )  \, d\mu_{\delta} +\frac{r^4}{128 \pi}  \int_{S_{r}(q)}   2 \mathcal{D}W_2\left( \delta( e_i, \nu ), \delta( r\xi-q, \nu ) \right) d\mu_{\delta} + \mathcal{O}(r^{-1})
    \end{split}
\end{equation}
where the last term is defined in (\ref{secondvary}). Note that by the parity conditions, the first and third terms are bounded, and by the decay conditions the second and fourth terms are bounded. Then if the right-hand side of the equality converges we have that the center of the foliation converges.

With this construction, we can obtain more explicit ways to calculate the center of the foliation, without the necessity to know the vectors $\xi(r)$ of the foliation.
\begin{corollary}
\label{corocenter1}     Let $(M,g,k)$ be $ \mathcal{C}^4$-asymptotic to Schwarzschild with mass $m>0$ satifying the condition of Theorem \ref{exisfoli}. Furthermore,  suppose that there is a point $q \in \mathbb{R}^3$ which satisfies. 
\begin{equation}\label{strongdec}
     x^i \partial_i (| x|^2 \mathrm{Sc})(x) = \mathcal{O}( | x|^{-3}),
\end{equation}
\begin{equation}\label{strongdecpi}
\pi (\Tilde{\rho}, \Tilde{\rho}) = o( | \Tilde{x}|^{-2}) , \quad (\nabla_{\Tilde{\rho}} \pi) (\Tilde{\rho}, \Tilde{\rho})= o(| \Tilde{x}|^{-3}) , \quad k(\cdot, \Tilde{\rho}) = o(| \Tilde{x}|^{-2}),
\end{equation}
\begin{equation}
    |k^{\text{even}}|(\Tilde{x}) + |\Tilde{x}||(\nabla k)^{\text{odd}}|(\Tilde{x}) = \mathcal{O}(|\Tilde{x}|^{-3}).
\end{equation}
Where $ \Tilde{x}= x-q$, $\Tilde{\rho}= \frac{\Tilde{x}}{|\Tilde{x}|}$. Also, suppose that the following limit converges  
\begin{equation}
    \lim_{r \to \infty}\frac{r^3}{128 \pi} \int_{S_r(q)} \mathrm{Sc} \nu^j + r \pi(\nu, \nu)   \nabla_{\nu}\pi(\nu,\nu) \nu^j +2  \pi(\nu, \nu)    k_i^j \nu^i 
   +    \pi(\nu, \nu)^2 \nu^j d\mu_{\delta} = Z^j
\end{equation}
Then there exists a foliation of Hawking surfaces, whose center is given by $C_H+Z$, where $C_H$ is the Hamiltonian center of mass.
\end{corollary}
\begin{proof}
Note that since $q$ is just finite vector in $\mathbb{R}^3 $,   $\mathcal{O}(|\Tilde{x}|^{-\alpha})= \mathcal{O}(|x|^{-\alpha})$ for any $\alpha>0$ and that   $x^i \partial_i (| x|^2 \mathrm{Sc})(x) = \mathcal{O}( | x|^{-3})$ can be integrated to get $\mathrm{Sc} = \mathcal{O}( | x|^{-5})$ ( $ |x|^2 \mathrm{Sc} = - \int_{|x|}^\infty  \frac{\partial}{\partial s} (s^2 \mathrm{Sc}) ds = - \int_{|x|}^\infty  \mathcal{O}(\eta_1 s^{-4})  ds = \mathcal{O}(\eta_1|x|^{-3})$ ), this implies in particular that 
\begin{equation}\label{extraesti}
        |\mathrm{Sc}( x)| +|x | |\nabla_{\rho }\mathrm{Sc}|  = \mathcal{O}(|x|^{-5}).
    \end{equation}
To see that $\xi = O(r^{-1})$, we proceed as in Lemma \ref{extixi}. Note that the condition (\ref{extraesti}) implies the assumptions on $\mathrm{Sc}$ considered in the lemma, therefore proceeding as its proof we obtain 
\begin{equation}
         |\xi|^{-1} \xi^i (\partial_i G_{r,2})(\xi) =  - \mathcal{O} (r^{-1}).
\end{equation}
 We also have 
\begin{equation}
\begin{split}
  |\xi|^{-1} \xi^i (\partial_i G_{r,3})(\xi) =& -r^2 \int_{S_{r}(q)}   2 W_2 \alpha d\mu_\delta -  r^2\int_{S_{r}(q)}   2 \mathcal{D}W_2\left(r \delta( \frac{\xi}{|\xi|}, \nu ),  \delta( r\xi-q, \nu ) \right) d\mu_\delta \\
 &- \frac{r^2}{2} \Bar{D}^2 \left(\int_{S_{r}(\xi r)}  2 W_2 \alpha d\mu_\delta\right)_{|\xi = \Tilde{\xi}}(r \xi-q, r \xi-q).\\
 =&  \mathcal{O}(r^{-1}) + o(r^{-1}). 
 \end{split}
\end{equation}
Where the fist term is estimated by $ \mathcal{O}(r^{-1})$ because of parity and the second term by $ o(r^{-1})$, since it depends on $\pi (\Tilde{\rho},\Tilde{\rho})$,  $\nabla_{\Tilde{\rho}} \pi (\Tilde{\rho}$, $ \Tilde{\rho})$ and  $k(\cdot, \Tilde{\rho}) $, and where $\Tilde{\xi}$ is a point such that $\Tilde{\xi}r$ lays in  between $\xi r$ and $q$. Then as in Lemma \ref{extixi} we have $\xi = O(r^{-1})$. 

Now proceeding as before  and using that $\frac{r^3}{128 \pi}  \int_{S_{r}(q)}   2 \mathcal{D}W_2\left( r \delta( \frac{\xi}{|\xi|}, \nu ),  \delta( r\xi-q, \nu )\right) d\mu_{\delta}=o(1) $ we have 
\begin{equation}
    \begin{split}
        r \delta(  \xi(r),& e_i ) =  \delta(  C_H, e_i )+ \frac{r^3}{128 \pi}  \int_{S_r(q)} \mathrm{Sc} \, \delta( e_i, \nu ) + r \pi(\nu, \nu)   \nabla_{\nu}\pi(\nu,\nu) \delta( \nu, e_i) +2  \pi(\nu, \nu)    k (\nu, e_i) \\
           &+    \pi(\nu, \nu)^2 \delta( \nu, e_i)   \, d\mu_{\delta} + \frac{r^3}{128 \pi}  \int_{S_r(q)}  \Big(  \delta( e_i, \nu )  \delta(  r \xi -q, \nu ) \Bar{D}_{\nu} \mathrm{Sc} +3 \delta( e_i, \nu ) \delta(    r \xi -q, \nu )  \frac{\mathrm{Sc} }{r} \\
        & -\frac{\mathrm{Sc} }{r} \, \delta( e_i, r \xi -q ) \Big) d\mu_{\delta} +\frac{r^4}{128 \pi}  \int_{S_{r}(q)}   2 \mathcal{D}W_2\left( \delta( e_i, \nu ), \delta(    r \xi -q, \nu )  \right) d\mu_{\delta} + \mathcal{O}(r^{-1}) \\
       &\, \textcolor{white}{11} = \delta(  C_H +Z, e_i )+ \mathcal{O}(r^{-1}) + o(1)
    \end{split}
\end{equation}
Then we have 
\begin{equation}
  r \xi(r) = C_H + Z +o(1)
\end{equation}
\end{proof}
\begin{remark}
    As mentioned in the proof, the assumption $x^i \partial_i (| x|^2 \mathrm{Sc})(x) = \mathcal{O}( | x|^{-3})$ is equivalent to  $|\mathrm{Sc}( x)| +|x | |\nabla_{\rho }\mathrm{Sc}|  = \mathcal{O}(|x|^{-5})$. Also note that if we furthermore assume  $ |k^{\text{even}}|(\Tilde{x}) + |\Tilde{x}||(\nabla k)^{\text{odd}}|(\Tilde{x}) = o(|\Tilde{x}|^{-3}) $ and $ |\mathrm{Sc}^{\text{odd}} |(\Tilde{x})=o ( |\Tilde{x}|^{-5})$ then $Z=0$ and the center converges to the  Hamiltonian center of mass. 
\end{remark}
If we don't consider the assumptions on $\pi$ we obtain.
\begin{corollary}\label{corocenter2}
     Let $(M,g,k)$ be $ \mathcal{C}^4$-asymptotic to Schwarzschild with mass $m>0$ satisfying the condition of Theorem \ref{exisfoli}. Furthermore,  suppose that there is a point  $q \in \mathbb{R}^3$ which satisfies. 
\begin{equation}\label{strongdec2}
     x^i \partial_i (| x|^2 \mathrm{Sc})(x) = \mathcal{O}( | x|^{-3}),
\end{equation}
\begin{equation}
    |k^{\text{even}}|(\Tilde{x}) + |\Tilde{x}||(\nabla k)^{\text{odd}}|(\Tilde{x}) = o(|\Tilde{x}|^{-3}).
\end{equation}
Where $ \Tilde{x}= x-q$. Also, suppose that the following limit converges  
\begin{equation}
    \lim_{r \to \infty}\frac{r^3}{128 \pi} \int_{S_r(q)} \mathrm{Sc} \, \nu^j +   2r \mathcal{D}W_2\left( \delta( e_j, \nu ),  \delta( r\xi-q, \nu ) \right) d\mu_{\delta}  = Z^j
\end{equation}
Then there exists a foliation of Hawking surfaces, whose  center is given by $C_H+Z$, where $C_H$ is the Hamiltonian center of mass. 
\end{corollary}
\begin{proof}
    The proof is similar to the one of Corollary \ref{corocenter1}, but this time $  r^3 \int_{S_{r}(q)}   2 W_2 \alpha d\mu_\delta =o(1)$, by the parity conditions.
\end{proof}
\begin{remark}
Note that in Theorem \ref{willcenmain} the center of symmetry of the previous two results $q$ would be the Hamiltonian center of mass $C_H$ however in case $k\neq 0$  the conditions 
    \begin{equation}
    \mathrm{Sc}(\tilde{x}) - \mathrm{Sc}(-\tilde{x}) = o(|x|^{-5}), \quad
  \tilde{x}^i \partial_i(|\tilde{x}|^2 \mathrm{Sc} (\tilde{x})) \leq o(|x|^{-3}),  
\end{equation}
are not enough to eliminate the $\mathrm{Sc} $ dependence on the center of the foliation, since the argument to eliminate the term containing $\mathrm{Sc} $ relies on the fact of not having other higher-order terms in the expression of $r \xi$ besides $\frac{r^3}{128 \pi} \int_{S_r(q)} \mathrm{Sc} \, \nu^j d\mu_{\delta}$ and $C_H$.
\end{remark}

\begin{remark}[Comparison of the Hawking energy]
\label{rmk:EH-comparison}
It follows from the decay of the first variation of the Hawking functional that
the Hawking energy of the leaves of the foliation can be estimated by that of
coordinate spheres. More precisely, one has
\begin{equation}
    \mathcal{E}(S_r(r\xi)) - \mathcal{E}(\Sigma(r)) = \mathcal{O}(r^{-2}) \quad \text{and} \quad  \mathcal{E}(S_r(0)) - \mathcal{E}(S_r(r\xi)) = \mathcal{O}(|\xi|),
\end{equation}
so that
\begin{equation}
    \mathcal{E}(\Sigma(r)) = \mathcal{E}(S_r(0)) + \mathcal{O}(|\xi|).
\end{equation}
Under the assumptions of Section~\ref{large foliations}, the translation
parameter $|\xi|$ is controlled by the decay of the scalar curvature and of
the tensor $k$.  In particular, under the stronger asymptotic conditions of
Section~\ref{sectioncenter}, one has $|\xi| = \mathcal{O}(r^{-1})$, and hence
   \begin{equation}
        \mathcal{E}(\Sigma(r)) = \mathcal{E}(S_r(0)) + \mathcal{O}(r^{-1}).
\end{equation}
\end{remark}
\section{Nonnegativity, monotonicity and rigidity}\label{rigifoligene}

 In this section, we will assume the dominant energy condition, which is  given by 
\begin{equation}
    \mu \geq |J|
\end{equation}
where 
\begin{equation}
     2 \mu:=\mathrm{Sc} + (\tr k)^2 - |k|^2 \quad \text{and} \quad  J:=\diver (k -(\tr k) g)
\end{equation}
are the energy density and the momentum density of the Einstein constraint equations.

\subsection{Nonnegativity along the foliation}

 In \cite{rigidiaz}, the following nonnegativity and rigidity result for Hawking surfaces was obtained. 
\begin{theorem}[{\cite[Theorem 3.6]{rigidiaz}}]\label{positivity0}
    Let $(M,g,k)$ be a $3$-dimensional initial data set satisfying the dominant energy condition and let $\Sigma$ be a Hawking surface with positive mean curvature, and such that for 
    \begin{equation}
        f:= \left( \frac{P}{H}\right)^2|k|^2+ \frac{1}{2 }(\tr k)^2    - \frac{3}{4} P^2- \frac{P}{H}( \nabla_\nu \tr k - \nabla_\nu k(\nu,\nu )) -  \frac{1}{2} |k|^2  -\frac{1}{2} |\mathring{B}|^2 -|J|
    \end{equation}
    the surface satisfies $\int_\Sigma f -\lambda d\mu \leq 0$. Then $\int_\Sigma H^2 -P^2 d\mu \leq 16\pi $, and  if $\int_\Sigma f -\lambda d\mu < 0$ then  $\int_\Sigma H^2 -P^2 d\mu < 16\pi $. In particular, the Hawking energy is nonnegative.

If additionally  $\Sigma $ is the boundary of a relatively compact domain  and there exists a constant $0\leq\beta <\frac{1}{2}$ such that $\int_\Sigma f_\beta -\lambda \, d\mu \leq 0$  for 
 \begin{equation}
        f_\beta:= \left( \frac{P}{H}\right)^2|k|^2+ \frac{1}{2 }(\tr k)^2   - \frac{3}{4} P^2- \frac{P}{H}( \nabla_\nu \tr k - \nabla_\nu k(\nu,\nu ))  -  \beta( |k|^2 + |\mathring{B}|^2 +2|J|).
    \end{equation}
Then if $\int_\Sigma H^2 -P^2 d\mu = 16\pi $, then   $\Omega$ is isometric to a spacelike hypersurface in Minkowski
spacetime with second fundamental form  $k$, $\Sigma$ is an umbilic round sphere and $k=0$ on $\Sigma$.
\end{theorem}
As mentioned in the remark following \cite[Theorem 3.6]{rigidiaz}   one could define $f$ differently, 
\begin{equation}\label{ftil}
    \Tilde{f}:= \frac{2P}{H}  k(\nabla^\Sigma \log H, \nu)+ \frac{1}{2 }(\tr k)^2    - \frac{3}{4} P^2- \frac{P}{H}( \nabla_\nu \tr k - \nabla_\nu k(\nu,\nu )) -  \frac{1}{2} |k|^2  -\frac{1}{2} |\mathring{B}|^2 -|J|.
\end{equation}
and one would also obtain nonnegativity of the Hawking energy. The same argument applies if one replaces $ \Tilde{f}$ by an  analogous   $ \Tilde{f}_{\beta}$, yielding an identical rigidity conclusion. We consider $f$ for computation simplicity.

First, we note that the Hawking energy is positive both along the foliation of Theorem~\ref{exisfoli} and for the on–center family of Theorem~\ref{family}. This follows directly from \cite[Theorem~2.2]{diaz2023local} (which treats the foliation) or from a direct application of Theorem~\ref{positivity0}.
\begin{theorem}\label{positivity}
Let $(M,g,k)$ be $ \mathcal{C}^4$–asymptotic to Schwarzschild and satisfy the dominant energy condition. Then the Hawking energy is positive on every leaf of the foliation of Theorem~\ref{exisfoli} and on every surface in the on–center family of Theorem~\ref{family}.
\end{theorem}
\begin{proof}
    We will apply Theorem \ref{positivity0}, therefore we need to see that $\int_\Sigma f-\lambda d\mu < 0$, which would imply positivity. By  Lemma \ref{expansquafoli}  we have $\lambda =\frac{4}{r^3} +\mathcal{O}(r^{-4}) $, in this lemma  $m$ was normalized to be equal to $2$, taking this into account we have  $0<\lambda =\frac{2m}{r^3} +\mathcal{O}(r^{-4})$. We also have that  $ H= \frac{2}{r} +\mathcal{O}(r^{-2})$ then by the decay conditions of $k$  the term $-\lambda$ dominates and it holds  $f-\lambda < 0$
\end{proof}
\subsection{Monotonicity}

Monotonicity lies at the very heart of what one expects from a physically meaningful quasi‐local energy: as a family of surfaces expands outward in an initial data set satisfying the dominant energy condition, its measured energy should not decrease. One of the most celebrated applications of monotonicity properties is in the proof of the Riemannian Penrose inequality by Huisken and Ilmanen \cite{huisken2001inverse}. In their  work, they employed the inverse mean curvature flow to show that the Hawking energy is nondecreasing along the flow under suitable conditions, such as nonnegative scalar curvature. This monotonicity, in the time-symmetric initial data setting, allowed them to rigorously relate the ADM mass of an asymptotically flat manifold to the area of its outermost apparent horizon.

Since then, there have been several generalizations of this monotonicity for the Hawking energy to different settings. For instance, Bray, Jauregui, and Mars in \cite{bray2015time,bray2016time} established monotonicity results for the spacetime setting for "time flat" surfaces, while more recently, Hirsch studied monotonicity in a general initial data set setting \cite{hirsch2022hawking}.

We aim to show that the Hawking energy is monotonically increasing along the previously constructed foliation. This property is already established in the totally geodesic case by the following result.
\begin{theorem}[{\cite[Theorem 4]{willflat}}]
    If \((M, g)\) satisfies \( \text{Sc} \geq 0 \) and if \( \Sigma \) is a compact spherical area-constrained Willmore surface  with \( H > 0 \), then \( m_H (\Sigma) \geq 0 \) if \( \lambda \geq 0 \). Furthermore, if \( F : \Sigma \times [0, \varepsilon) \rightarrow M \) is a variation with initial velocity \( \frac{\partial F}{\partial s} \Big|_{s=0} = \alpha \nu \) and $\int_{\Sigma} \alpha H \, d\mu \geq 0$,
then 
\[
\frac{d}{ds}  \mathcal{E}(F(\Sigma, s)) \geq 0.
\]
Note that the condition on \( \alpha \) means that the area is increasing along the variation. In particular the Hawking energy is monotonically increasing along the Willmore foliation.
\end{theorem}

We want to study what happens in the general case under the dominant energy condition.
\begin{theorem}\label{monotinic0}
   Let $(M,g,k)$ be an initial data set  satisfying the dominant energy condition and let \( \Sigma \) be a Hawking surface with positive mean curvature satisfying  $\int_\Sigma f \, d\mu \leq 0$ for
    \begin{equation}
        f:= \left( \frac{P}{H}\right)^2|k|^2+ \frac{1}{2 }(\tr k)^2    - \frac{3}{4} P^2- \frac{P}{H}( \nabla_\nu \tr k - \nabla_\nu k(\nu,\nu )) -  \frac{1}{2} |k|^2  -\frac{1}{2} |\mathring{B}|^2 -|J|.
    \end{equation}
     Then if \( F : \Sigma \times [0, \varepsilon) \rightarrow M \) is a variation with initial velocity \( \frac{\partial F}{\partial s} \Big|_{s=0} = \alpha \nu \) and $\int_{\Sigma} \alpha H \, d\mu \geq 0,$ then 
\[
\frac{d}{ds} \mathcal{E}(F(\Sigma, s)) \geq 0.
\]
\end{theorem}
\begin{proof}
We proceed as in \cite[Theorem 4]{willflat}. We compute the variation of the Hawking energy
\begin{equation}
    \begin{split}
        (16 \pi)^{3/2} \frac{d}{ds}  \mathcal{E}(\Sigma_s)_{|s=0}&= \frac{1}{2|\Sigma|^{1/2} }\left(\int_\Sigma H \alpha d\mu \right) \left(16 \pi -  \int_\Sigma H^2-P^2 d\mu \right)-  2 |\Sigma|^{1/2} \int_\Sigma \lambda H \alpha d\mu\\
        &=\frac{1}{2|\Sigma|^{1/2} }\left(\int_\Sigma H \alpha d\mu \right) \left(16 \pi - 4 \lambda |\Sigma| - \int_\Sigma H^2-P^2 d\mu \right)\\
    \end{split}
\end{equation}
Then  we need to show that $\lambda |\Sigma| + \frac{1}{4}\int_\Sigma H^2-P^2 d\mu \leq 4 \pi$. 
We consider  equation (\ref{eulag}), divide it  by $H$, integrate by parts the term $\frac{\Delta H}{H}$ and use the Gauss equation $2\mathrm{Ric}(\nu, \nu) = \mathrm{Sc} -\mathrm{Sc}^{\Sigma_r} + H^2 -|B|^2  $ obtaining
\begin{equation*}
\begin{split}
   0=  \int_{\Sigma_r}&\lambda+ |\nabla^\Sigma  \log H|^2 + \frac{1}{2} |\mathring{B}|^2+  \frac{1}{2}(\mathrm{Sc} -\mathrm{Sc}^{\Sigma_r}) +\frac{P}{H}( \nabla_\nu \tr k - \nabla_\nu k(\nu,\nu ))\\& +\frac{1}{4}H^2 + \frac{1}{2} P^2 - \frac{2P}{H}  k(\nabla^\Sigma \log H, \nu)d\mu. 
    \end{split}
\end{equation*}
Now using  Gauss-Bonnet theorem to replace $\mathrm{Sc}^{\Sigma_r}$,  adding and subtracting $(\tr k)^2$, $|k|^2$ and $|J|$ we have 
\begin{equation*}
\begin{split}
   \lambda |\Sigma|+  \frac{1}{4}\int_\Sigma H^2-P^2 d\mu &= 4\pi +\int_\Sigma \frac{(\tr k)^2}{2 } - \frac{|k|^2}{2} -\frac{P}{H}( \nabla_\nu \tr k - \nabla_\nu k(\nu,\nu ))  - \frac{|\mathring{B}|^2}{2}   - \frac{3}{4} P^2 -|J| \\
   &   -|\nabla^\Sigma  \log H|^2 + \frac{2P}{H}  k(\nabla^\Sigma \log H, \nu) -\frac{1}{2}(\mathrm{Sc} + (\tr k)^2 - |k|^2-2|J|) d\mu.\\
    &=4\pi +\int_\Sigma  f+g -(\mu-|J|)\, d\mu
    \end{split}
\end{equation*}
where 
$$g= -\left( \frac{P}{H}\right)^2|k|^2 -|\nabla^\Sigma  \log H|^2 + \frac{2P}{H}  k(\nabla^\Sigma \log H, \nu). $$
It is direct to see that $g\leq 0$, then we have that the integral is nonpositive and the result follows.
\end{proof}
\begin{remark}
    Here again, just as in Theorem \ref{positivity0}, one may replace $f$ by 
    $ \Tilde{f}:= \frac{2P}{H}  k(\nabla^\Sigma \log H, \nu)+ \frac{1}{2 }(\tr k)^2    - \frac{3}{4} P^2- \frac{P}{H}( \nabla_\nu \tr k - \nabla_\nu k(\nu,\nu )) -  \frac{1}{2} |k|^2  -\frac{1}{2} |\mathring{B}|^2 -|J|$ and the same conclusion holds.
\end{remark}
Then with this result, we have the monotonicity of the foliation.
\begin{corollary}
    Let $(M,g,k)$ be an asymptotically flat  initial data set satisfying the dominant energy condition. If the initial data set possesses a foliation by  Hawking surfaces and   $\int_{\Sigma_{s}} f \, d\mu \leq 0$ for all the surfaces of the foliation, then the foliation is monotonically increasing.
\end{corollary}
Note that  the condition $\int_{\Sigma_{s}} f \, d\mu \leq 0$ is not satisfied by every critical surface of the Hawking functional. However, as we will see next, this condition is quite general, as it holds for a large class of initial data sets.

\subsection*{Harmonic asymptotics}

An asymptotically flat initial data set \((M, g, k)\) has \emph{harmonic asymptotics} if, outside a compact set, the metric \( g \) and canonically conjugated momentum \( \pi \) can be expressed as:
\begin{equation}\label{metricharm}
g = u^4 \delta, \quad \pi = u^2 \mathcal{L} X,
\end{equation}
where \( u \) is a scalar function, \( X \) is a vector field, \( \delta \) is the Euclidean metric, and \( \mathcal{L} X \) is defined as:
\[
\mathcal{L} X = L_X \delta - \text{div}_\delta(X) \,\delta,
\]
with \( L_X \delta \) denoting the Lie derivative. Recall that \( \pi = k - \tr k \, g \). 

 Furthermore, the scalar function $u$ and the vector field $X$ then satisfy the asymptotics 
\begin{equation}\label{harmoasym}
    u(x) = 1 + \frac{m}{2|x|} + \mathcal{O} (|x|^{-2}) \quad \text{and} \quad X^i = -\frac{2p^i}{|x|} + \mathcal{O} (|x|^{-2})
\end{equation}
These are the so-called harmonic asymptotics. Substituting these expressions into the definition of \( \pi \), the leading-order term becomes:
\begin{equation}\label{momencor}
\pi = \frac{2}{|x|^2} \left( p \otimes \rho + \rho \otimes p - \delta( p, \rho) \,\delta \right) + \mathcal{O}(|x|^{-3}),
\end{equation}
where  $\rho = x/|x|$ is the radial direction. The ADM energy of the initial data set is given by the parameter $m$ (see Remark \ref{remarkm}), and its ADM momentum
\[
p_{ADM}(e_i) := \lim_{r \to \infty} \frac{1}{8\pi} \int_{S_r} \pi (e_i, \nu) \, d\mu_e,
\]
is given by $p$. Note that an asymptotically Schwarzschild metric has in particular harmonic asymptotics. 

 Initial data sets with harmonic asymptotics are among the most  significant in general relativity, since they are dense in the space of all solutions to the asymptotically flat constraint equations satisfying the dominant energy conditions. In Appendix \ref{appendix} we give precise statements of this result. Therefore, it is especially compelling to prove monotonicity within this class of initial data sets.

We assume that we have an initial data set with harmonic asymptotics that satisfies the dominant energy condition. Furthermore, assume that the manifold possesses a critical surface of the Hawking energy and that the surface can be put as graphs over a sphere centered at $r \xi$ with $\xi=O(\eta)$, for $\eta$ small, we will see later that $\eta$ must be smaller than the ADM momentum of the initial data set.  Such surfaces are present in the foliation constructed in Theorem \ref{exisfoli}. In particular, one can express the normal to the surfaces as $\nu= \rho +\mathcal{O}(\eta)$.

By (\ref{momencor}) and $ k= \pi - \frac{1}{2} \tr \pi \,g$, we see that the tensor $k$  has the form
\[
k = \frac{2}{|x|^2} (p \otimes \rho + \rho \otimes p -\frac{1}{2} \delta( p, \rho )\, \delta) + \mathcal{O} ( |x|^{-3}) 
\]
and we have the following quantities.
$$ \tr k = \frac{1}{|x|^2} \delta( p, \rho ) +\mathcal{O} (|x|^{-3}), \quad |k|^2 = \frac{4}{|x|^4} (2|p|^2 + \frac{3}{4} \delta( \rho, p )^2 ) +\mathcal{O} (|x|^{-5})$$
it also holds
$$\diver (\frac{2}{|x|^2} (p \otimes \rho + \rho \otimes p - \delta( p, \rho )\, \delta)) =0 $$
  and then 
$$|J|= \mathcal{O} (|x|^{-4}) $$
Note that since we are assuming  the dominant energy condition it holds  
$$\mathrm{Sc} + (\tr k)^2 - |k|^2= \mathrm{Sc} +\frac{1}{|x|^4}\delta( p,\rho )^2 - \frac{4}{|x|^4} (2|p|^2 + \frac{3}{4} \delta( \rho, p )^2 ) + \mathcal{O} ( |x|^{-5}) \geq 2 |J| $$
 Also taking $\nu= \rho +\mathcal{O}(\eta)$ to be the normal to the surface 
$$P=-\pi(\nu, \nu)=-\frac{2}{|x|^2} \delta( p, \rho ) +\mathcal{O} (\eta |x|^{-2})  $$
$$ \nabla_\nu \tr k = -\frac{2}{|x|^3} \delta( p, \rho ) + \mathcal{O} (\eta |x|^{-3}),\quad \nabla_\nu k(\nu, \nu)= -\frac{6}{|x|^3} \delta( p, \rho ) + \mathcal{O} ( \eta|x|^{-3}) $$
and 
$$\frac{P}{H}( \nabla_\nu \tr k - \nabla_\nu k(\nu,\nu )) = \frac{4}{|x|^4} \delta( p, \rho )^2 +  \mathcal{O} ( \eta |x|^{-4 }) .$$
Using the previous quantities and  taking $\eta \ll p$ then
\begin{equation}\label{monpiece}
\begin{split}
        f&=\left( \frac{P}{H}\right)^2|k|^2+ \frac{1}{2 }(\tr k)^2    - \frac{3}{4} P^2- \frac{P}{H}( \nabla_\nu \tr k - \nabla_\nu k(\nu,\nu )) -  \frac{1}{2} |k|^2  -\frac{1}{2} |\mathring{B}|^2 -|J| \\
        &\leq  \frac{1}{2 }(\tr k)^2 -  \frac{1}{2} |k|^2    - \frac{3}{4} P^2- \frac{P}{H}( \nabla_\nu \tr k - \nabla_\nu k(\nu,\nu )) +\mathcal{O} ( |x|^{-6})\\  
        &=   -\frac{4}{|x|^4}|p|^2 -\frac{8}{|x|^4}  \delta( p, \rho )^2 +\mathcal{O} (\eta |x|^{-4}) \leq 0    
\end{split}
\end{equation}
Then we have the following result, which implies the monotonicity of the Hawking energy along the foliation.
\begin{theorem}\label{monocosho}
    Let $(M,g,k)$ be an initial data set with harmonic  asymptotics, satisfying the dominant energy condition and with ADM linear momentum $p_{ADM} \neq 0$. Suppose \(\Sigma\) is a spherical Hawking surface with \(H>0\) that can be written as a graph over a coordinate sphere of radius \(r\) with center \(r\,\xi\), where \(\xi=O(\eta)\) and \(\eta\ll p_{\mathrm{ADM}}\). If \( F : \Sigma \times [0, \varepsilon) \rightarrow M \) is any variation with \( \frac{\partial F}{\partial s} \Big|_{s=0} = \alpha \nu \) and $\int_{\Sigma} \alpha H \, d\mu \geq 0,$ then 
\[
\frac{d}{ds} \mathcal{E}(F(\Sigma, s)) \geq 0.
\]
\end{theorem}
\begin{remark}
Since initial data sets with harmonic asymptotics are dense among solutions of the Einstein constraint equations under the dominant energy condition (see Theorem~\ref{densitytheo} and Appendix~\ref{appendix}), it is tempting to prove monotonicity by approximation: verify the sign of the integrand for the harmonic model and pass to the limit. Recall that
\[
f(k,\nabla k)
:=\Bigl(\frac{P}{H}\Bigr)^{\!2}|k|^2
+\tfrac12(\tr k)^2
-\tfrac34 P^2
-\frac{P}{H}\bigl(\nabla_\nu \tr k-\nabla_\nu k(\nu,\nu)\bigr)
-\tfrac12|k|^2
-\tfrac12|\mathring B|^2
-|J|.
\]
If \(f\le0\) on \(\Sigma\), then the Hawking energy is nondecreasing along the foliation. In \eqref{monpiece} we showed that for harmonically asymptotic data one indeed has \(f<0\).

Let us explain why density alone does not propagate this sign. By Theorem~\ref{densitytheo} and the weighted Sobolev embedding(Theorem~\ref{sobolevemb}, using the decays \(k=\mathcal O(|x|^{-2})\) and \(g=\mathcal O(|x|^{-1})\) and taking \(p\) large), for every \(\varepsilon>0\) and every \(q\in(1/2,1)\) there exists \((\tilde g_{\varepsilon,q},\tilde k_{\varepsilon,q})\) with harmonic asymptotics such that
\begin{equation}\label{decayreg}
\|g-\tilde g_{\varepsilon,q}\|_{\mathcal C^{2,\alpha}_{-q}}<\varepsilon,
\qquad
\|k-\tilde k_{\varepsilon,q}\|_{\mathcal C^{1,\alpha}_{-1-q}}<\varepsilon,
\end{equation}
and \(|E-\tilde E_{\varepsilon,q}|<\varepsilon\), \(|p_{ADM}-\tilde p_{\varepsilon,q,ADM}|<\varepsilon\). From \eqref{decayreg} one obtains, for some \(C>0\),
\[
\bigl|f(k,\nabla k)-f(\tilde k_{\varepsilon,q},\nabla\tilde k_{\varepsilon,q})\bigr|
\;\le\; \frac{C\varepsilon}{|x|^{3+q}}
\qquad\text{for } \tfrac12<q<1.
\]
Assuming for simplicity that \(\nu=\rho+\mathcal{O}(|x|^{-1})\), the harmonic expansion \eqref{monpiece} yields
\begin{equation}
    \begin{split}
        f(k,\nabla k) &=f(\tilde{k}_{\epsilon,q},\nabla \tilde{k}_{\epsilon,q})  + \mathcal{O}(\frac{  \epsilon}{|x|^{3+q}})\\
        &=  -\frac{4}{|x|^4}|p_{\epsilon,q}|^2 -\frac{8}{|x|^4}  \delta( p_{\epsilon,q}, \rho )^2 -\frac{1}{2} |\mathring{B}|^2 +\mathcal{O}_{\epsilon,q} ( |x|^{-5}) + \mathcal{O}(\frac{  \epsilon}{|x|^{3+q}})
    \end{split}
\end{equation}
Here \(\mathcal O_{\varepsilon,q}\) indicates dependence of the constant on \(\varepsilon,q\). The leading negative contribution is of order \(|x|^{-4}\), whereas the approximation error is of order \(\varepsilon\,|x|^{-3-q}\), which is \emph{larger} in magnitude for \(q<1\); hence it can dominate the negative term. Moreover, trying to suppress this error by taking \(\varepsilon\) small and \(q\to1^{-}\) may enlarge the hidden constant in \(\mathcal O_{\varepsilon,q}(|x|^{-5})\), potentially promoting it to an \(\mathcal O(|x|^{-4})\) contribution. One could further expand the term   $f(\tilde{k}_{\epsilon,q},\nabla \tilde{k}_{\epsilon,q})$ to make the next term of the expansion explicit, but one would have the same problem with the error term of the new expansion. Consequently, the sign of \(f(k,\nabla k)\) cannot be concluded from density alone, and a separate monotonicity argument is required.
\end{remark}
\subsection*{York asymptotics}
In the early development of the conformal method for constructing solutions to the Einstein constraint equations, York studied the asymptotically flat case in \cite{york1979kinematics} and proposed the following model for the extrinsic curvature:
\[
k = \frac{3}{2|x|^2} \left( \rho \otimes p + p \otimes \rho - \delta( p, \rho ) (\delta - \rho \otimes \rho) \right) + \mathcal{O}(|x|^{-3}),
\]
where \( p \in \mathbb{R}^3 \) is a fixed vector representing the ADM linear momentum and \( \rho = \frac{x}{r} \) is the radial direction.

This model was derived under two key conditions:
\begin{enumerate}
    \item The ADM linear momentum of the manifold is given by \( p \).
    \item The tensor \( k \) is asymptotically maximal, meaning \( \operatorname{tr} k = \mathcal{O}(|x|^{-3}) \).
\end{enumerate}
Both conditions are satisfied, as the leading term of \( k \) is trace-free with respect to both the Euclidean and Schwarzschild metrics. 
\begin{remark}
    When deriving this model in Chapter 6 of York's paper \cite{york1979kinematics}, the transverse-traceless (TT) part of the tensor \( A \), denoted by \( A^* \), is not explicitly determined during the construction of the model for asymptotically flat manifolds. Instead, York focuses on the decomposition of \( A \) into its TT part and a divergence-free component derived from the vector potential \( W \):
\[
A_{ij} = A^*_{ij} + (LW)_{ij},
\]
where \( (LW)_{ij} = \nabla_i W_j + \nabla_j W_i - \frac{2}{3} g_{ij} \nabla^k W_k \).

The TT part \( A^*_{ij} \) is implicitly treated as prescribed data that satisfies the following conditions:
\[
\nabla^j A^*_{ij} = 0 \quad \text{and} \quad g^{ij} A^*_{ij} = 0,
\]
ensuring it is both divergence-free and trace-free. This part of the decomposition must also adhere to the asymptotic conditions required for asymptotically flat spacetimes, such as appropriate fall-off rates at spatial infinity.

York does not provide an explicit method for fixing \( A^*_{ij} \) within this chapter, as its determination depends on the specific context or physical scenario being modeled. Instead, the main focus lies on solving for the scalar \( \phi \) (the conformal factor) and the vector \( W \), while \( A^* \) acts as a prescribed input satisfying the necessary divergence-free and asymptotic conditions.
\end{remark}
We assume  an initial data set $(M,g,k)$ with $g=\delta + \mathcal{O}(|x|^{-1}) $, $k$ with York asymptotics  and  which satisfies the dominant energy condition. Furthermore, assume as before that the manifold possesses a critical surface of the Hawking energy, and that the surface can be put as graphs over a sphere centered at $r \xi$ with $\xi=O(\eta)$, for $\eta$ small. In particular one can express the normal to the surfaces as $\nu= \rho +O(\eta)$.

Then we have 
$$\tr k = \mathcal{O}(|x|^{-3}), \quad \pi = k +\mathcal{O}(|x|^{-3})  $$
also 
$$ \diver (\frac{3}{2|x|^2} \left( \rho \otimes p + p \otimes \rho - \delta( p, \rho ) (\delta - \rho \otimes \rho) \right))=0$$
then
$$|J|= | \diver (\pi)|= \mathcal{O} ( |x|^{-4})$$
and
$$|k|^2= \frac{9}{2 |x|^4} (|p|^2 +2 \delta( p, \rho )^2) + \mathcal{O} (|x|^{-5}).$$
Note that  the dominant energy condition  is given by 
$$\mathrm{Sc} + (\tr k)^2 - |k|^2= \mathrm{Sc}- \frac{9}{2 |x|^4} (|p|^2 +2 \delta( p, \rho )^2) + \mathcal{O} (|x|^{-5}) \geq  2|J|. $$
 Also taking $\nu= \rho +O(\eta)$ to be the normal to the surface  we have 
$$P=-k(\nu,\nu)+ \mathcal{O}(|x|^{-3})=-\frac{3}{|x|^2} \delta( p, \rho ) +  \mathcal{O}(\eta |x|^{-2})$$
$$ \nabla_\nu k(\nu,\nu ) =\nabla_\rho k(\rho,\rho ) +  \mathcal{O}(\eta |x|^{-2})= -\frac{6}{|x|^3}  \delta( p, \rho ) +  \mathcal{O}(\eta|x|^{-3}) $$
and 
$$ \frac{P}{H}( \nabla_\nu \tr k - \nabla_\nu k(\nu,\nu )) =- \frac{9}{|x|^4} \delta( p, \rho )^2 +  \mathcal{O}(\eta|x|^{-4})$$
Using the previous quantities and  taking $\eta \ll p$ then
\begin{equation}
\begin{split}
        f&= \left( \frac{P}{H}\right)^2|k|^2+ \frac{1}{2 }(\tr k)^2    - \frac{3}{4} P^2- \frac{P}{H}( \nabla_\nu \tr k - \nabla_\nu k(\nu,\nu )) -  \frac{1}{2} |k|^2  -\frac{1}{2} |\mathring{B}|^2 -|J| \\
        &\leq   -  \frac{1}{2} |k|^2    - \frac{3}{4} P^2- \frac{P}{H}( \nabla_\nu \tr k - \nabla_\nu k(\nu,\nu )) +\mathcal{O} ( |x|^{-6})\\  
        &=   -\frac{9}{4|x|^4}(|p|^2 + \delta( p, \rho )^2 ) +\mathcal{O} (\eta |x|^{-4}) \leq 0    
\end{split}
\end{equation}
and then we have the following result similar to Theorem \ref{monocosho}  which implies the monotonicity of the Hawking energy along the foliation in the case of York asymptotics.
\begin{theorem}\label{monoyork}
    Let $(M,g,k)$ be an initial data set satisfying the dominant energy condition,  with $g=\delta + \mathcal{O}(|x|^{-1}) $, $k$ with York asymptotics and with ADM linear momentum $p_{ADM} \neq 0$. If \( \Sigma \) is a spherical Hawking surface  with \( H > 0 \) and which can be put as a graph over a sphere centered at $r \xi$ with $\xi=O(\eta)$, and $\eta \ll p_{ADM}$. Then if \( F : \Sigma \times [0, \varepsilon) \rightarrow M \) is a variation with initial velocity \( \frac{\partial F}{\partial s} \Big|_{s=0} = \alpha \nu \) and $\int_{\Sigma} \alpha H \, d\mu \geq 0,$ then 
$$ \frac{d}{ds} \mathcal{E} (F(\Sigma, s)) \geq 0. $$
\end{theorem}
\subsection{Rigidity on the foliation}

In \cite{rigidiaz}, the following rigidity result for the Wilmore foliation was obtained.
\begin{theorem} [{\cite[Theorem 2.16]{rigidiaz}}]  \label{rigidityfoli}
    Let \((M, g)\) be a $3$-dimensional asymptotically flat  Riemannian manifold with nonnegative  scalar curvature. Then the Hawking  and  Brown-York energies of all the  Willmore surfaces of the canonical Willmore foliation  are positive unless $(M,g)$ is isometric to Euclidean space.
\end{theorem}
This  result relies on the following result of Shi  for isoperimetric surfaces.
\begin{theorem}[{\cite[Theorem 3]{shi2016isoperimetric}}]\label{isoperig}
    Suppose \((M, g)\) is a complete  asymptotically flat manifold with nonnegative scalar curvature. Then for any  \(V > 0\),
\begin{equation}
I(V) \leq (36\pi)^{\frac{1}{3}} V^{\frac{2}{3}}. 
\end{equation}
There exists a  \(V_0 > 0\) such that
\begin{equation}
I(V_0) = (36\pi)^{\frac{1}{3}} V_0^{\frac{2}{3}} 
\end{equation}
if and only if \((M, g)\) is isometric to \(\mathbb{R}^3\). Here
\[
I(v) = \inf \left\{ \mathcal{H}^2(\partial^* \Omega) : \Omega \subset M \text{ is a Borel set with finite perimeter, and } \mathcal{L}^3(\Omega) = v \right\},
\]
is the isoperimetric profile, where \(\mathcal{H}^2\) is the 2-dimensional Hausdorff measure of the reduced boundary \(\partial^* \Omega\), and \(\mathcal{L}^3(\Omega)\) is the Lebesgue measure of \(\Omega\) with respect to the metric \(g\).
\end{theorem}
We now state a dynamical generalization of Theorem \ref{rigidityfoli}. Unlike our previous foliation theorems, this result does \emph{not} assume an asymptotically Schwarzschild end; it holds on any asymptotically flat initial data set admitting a foliation by Hawking spheres.
\begin{theorem}\label{rigigeneralfoli}
   Let $(M,g,k)$ be an asymptotically flat  initial data set satisfying the dominant energy condition. If the initial data set possesses an on-center foliation by Hawking surfaces   and one  surface $\Sigma$ of the foliation satisfies:

    \begin{enumerate}[label=(\roman*)]
    
    \item The Hawking energy of the surface is zero.
    
    \item Outside of $\Sigma$  $(M,g)$ has nonnegative scalar curvature. 

    \item There exists a constant $0 \leq\beta <\frac{1}{2}$ such that $\int_\Sigma f_\beta -\lambda \, d\mu \leq 0$  for 
 \begin{equation}
       f_\beta:= \left( \frac{P}{H}\right)^2|k|^2+ \frac{1}{2 }(\tr k)^2   - \frac{3}{4} P^2- \frac{P}{H}( \nabla_\nu \tr k - \nabla_\nu k(\nu,\nu ))  -  \beta( |k|^2 + |\mathring{B}|^2 +2|J|).
    \end{equation}
\end{enumerate}    
Then $(M,g,k)$ is isometric to a spacelike hypersurface in  Minkowski spacetime.
\end{theorem}
\begin{proof}
Denoting by $\Omega$ interior of $\Sigma$, we can directly apply Theorem \ref{positivity0} obtaining that $\Omega$  is isometric to a spacelike hypersurface in Minkowski spacetime,  also we have that $\Sigma$ is a round sphere with $|k|_{\Sigma} =\lambda=|\mathring{B}|= \mathrm{Sc}_{| \Sigma} =\mathrm{Ric}_{| \Sigma}(\nu, \nu)=0$ and $H=\frac{2}{r^2}$ is constant, where $r$ is the area radius of $\Sigma$ (for more details see \cite[Theorem 3.6]{rigidiaz}). Then what it remains to see is that $M \setminus \Omega$ is also isometric to a hypersurface in Minkowski space.

Now by considering $\Omega$ as a spacelike hypersurface in Minkowski spacetime,  using that its second fundamental form  $ k$ vanishes at $\partial \Omega = \Sigma$, that the Riemann tensor in Minkowski is zero and the Gauss equation for the curvature tensor, we find that the Riemann curvature tensor $ \mathrm{Rm}^M=0$ vanishes along $\partial \Omega$.  Then with this and the previously mentioned properties of $\Sigma$, we can glue $ M \setminus \Omega$ to a round ball of radius $r$, with this we have an asymptotically flat $ \mathcal{C}^2$ manifold  $B_r(0)\cup (M \setminus \Omega) $ with nonnegative scalar curvature. As $\Sigma$ is an isoperimetric surface, we can apply Theorem \ref{isoperig}  obtaining that $B_r(0)\cup (M \setminus \Omega) $ is isometric to Euclidean space. In particular,  $M \setminus \Omega $ is isometric to Euclidean space minus a ball. By the positive mass theorem $M \setminus \Omega $ has zero ADM energy and ADM momentum, and it is direct to see that $k=0$, since the only hypersurfaces in Minkowski spacetime isometric to Euclidean space are hyperplanes.  With this, we have  $M$ isometric to a spacelike hypersurface in Minkowski spacetime with $k=0$ outside of $\Omega$. 
\end{proof}
\begin{remark}\label{remarkrigidi}
$i)$    Note that by the dominant energy condition, the condition $\mathrm{Sc} \geq 0$ outside of $\Sigma$ is implied if $\tr k=0$ outside of $\Sigma$.

$ii)$ We only require asymptotic flatness here, a strictly weaker condition than asymptotically Schwarzschild. In the latter setting, Theorem \ref{positivity} already shows that the Hawking energy along our foliation can vanish only if the ADM mass is zero.
\end{remark}
 Finally, we will see that the conditions of this theorem are satisfied by the two models we saw. 
 \begin{corollary}\label{rigiharmoyor}
    Let $(M,g,k)$ be an asymptotically flat initial data set satisfying the dominant energy condition, with $k$ having either harmonic or York asymptotics and with ADM linear momentum $p_{ADM} \neq 0$.   If the initial data set possesses a large-scale foliation by  Hawking surfaces, and each leaf of the foliation can be put as a graph over a sphere centered at $r \xi$ with $\xi=O(\eta)$ and $\eta \ll p_{ADM}$. Then the  Hawking energy along the foliation cannot be zero.
\end{corollary}
\begin{proof}
Recall that $k$ has harmonic asymptotics if 
\[
k = \frac{2}{|x|^2} (p \otimes \rho + \rho \otimes p -\frac{1}{2} \delta( p, \rho ) \delta) + \mathcal{O} ( |x|^{-3}) 
\]
and that the   dominant energy condition tells us that  
$$\mathrm{Sc} + (\tr k)^2 - |k|^2= \mathrm{Sc} +\frac{1}{|x|^4}\delta( p,\rho )^2 - \frac{4}{|x|^4} (2|p|^2 + \frac{3}{4} \delta( \rho, p )^2 ) + \mathcal{O} ( |x|^{-5}) \geq 0 $$
which implies that if the ADM momentum of $k$ is nonzero then $\mathrm{Sc} >0$. Furthermore,  as seen when studying the monotonicity 
\begin{equation}
\begin{split}
        f&=\left( \frac{P}{H}\right)^2|k|^2+ \frac{1}{2 }(\tr k)^2    - \frac{3}{4} P^2- \frac{P}{H}( \nabla_\nu \tr k - \nabla_\nu k(\nu,\nu )) -  \frac{1}{2} |k|^2  -\frac{1}{2} |\mathring{B}|^2 -|J|\\
        &\leq  \frac{1}{2 }(\tr k)^2 -  \frac{1}{2} |k|^2    - \frac{3}{4} P^2- \frac{P}{H}( \nabla_\nu \tr k - \nabla_\nu k(\nu,\nu )) +\mathcal{O} ( |x|^{-6})\\  
        &=   -\frac{4}{|x|^4}|p|^2 -\frac{8}{|x|^4}  \delta( p, \rho )^2 +\mathcal{O} (\eta |x|^{-4}) \leq 0    
\end{split}
\end{equation}
Similarly recall that $k$ has York asymptotics if 
$$k = \frac{3}{2|x|^2} \left( \rho \otimes p + p \otimes \rho - \delta( p, \rho ) (\delta - \rho \otimes \rho) \right) +\mathcal{O}(|x|^{-3}) $$
and  the dominant energy condition  is given by 
$$\mathrm{Sc} + (\tr k)^2 - |k|^2= \mathrm{Sc}- \frac{9}{2 |x|^4} (|p|^2 +2 \delta( p, \rho )^2) + \mathcal{O} (|x|^{-5}) \geq  0. $$
This implies that if the ADM momentum of $k$ is nonzero then $\mathrm{Sc} >0$.
Also using that $\eta \ll p$ we have for the york asymptotics 
\begin{equation}
\begin{split}
        f&=\left( \frac{P}{H}\right)^2|k|^2+ \frac{1}{2 }(\tr k)^2    - \frac{3}{4} P^2- \frac{P}{H}( \nabla_\nu \tr k - \nabla_\nu k(\nu,\nu )) -  \frac{1}{2} |k|^2  -\frac{1}{2} |\mathring{B}|^2 -|J|\\
        &\leq   -  \frac{1}{2} |k|^2    - \frac{3}{4} P^2- \frac{P}{H}( \nabla_\nu \tr k - \nabla_\nu k(\nu,\nu )) +\mathcal{O} ( |x|^{-6})\\  
        &=   -\frac{9}{4|x|^4}(|p|^2 + \delta( p, \rho )^2 ) +\mathcal{O} (\eta |x|^{-4}) \leq 0    
\end{split}
\end{equation}
Then we can apply  Theorem \ref{rigigeneralfoli} for both cases and if the Hawking energy on a leaf of the foliation is zero then the initial data set is a hypersurface in Minkowski space implying that it has zero ADM energy and therefore zero ADM momentum, a contradiction.
\end{proof}

 \appendix

\section{Density of Harmonic asymptotics}\label{appendix}

 Initial data sets with harmonic asymptotics form a dense subset within the space of solutions to the asymptotically flat constraint equations. To state this property precisely, we will now recall some related results and definitions. 
\begin{definition}
    Let $B$ be a closed ball in $\mathbb{R}^n$ with center at the origin. For every $k \in \{0, 1, \dots\}$, $p \geq 1$, and $q \in \mathbb{R}$, we define the weighted Sobolev space $W^{k,p}_{-q}(\mathbb{R}^n \setminus B)$ as the collection of those $f \in W^{k,p}_{\text{loc}}(\mathbb{R}^n \setminus B)$ such that
\[
\| f \|_{W^{k,p}_{-q}(\mathbb{R}^n \setminus B)} :=
\left(
\int_{\mathbb{R}^n \setminus B}
\sum_{|I| \leq k}
\left( |\partial^I f(x)| |x|^{|I|+q} \right)^p |x|^{-n} \, dx
\right)^{\frac{1}{p}}
< \infty.
\]
Where $ W^{k,p}_{\text{loc}} $ denotes the set of local Sobolev functions.
\end{definition}
Suppose now that $M$ is a $\mathcal{C}^k$ manifold such that there is a compact set $K \subset M$ and a diffeomorphism $M \setminus K \cong \mathbb{R}^n \setminus B$. The $W^{k,p}_{-q}$ norm on $M$ is defined by choosing an atlas for $M$ that consists of the diffeomorphism $M \setminus K \cong \mathbb{R}^n \setminus B$ and finitely many precompact charts, and then summing the $W^{k,p}_{-q}(\mathbb{R}^n \setminus B)$ norm on the noncompact chart and the $W^{k,p}$ norms on the precompact charts. The resulting space $W^{k,p}_{-q}(M)$ and its topology only depend on the diffeomorphism $M \setminus K \cong \mathbb{R}^n \setminus B$.

Similarly, we can define a Hölder norm.
\begin{definition}
    \noindent Let $B$ be a closed ball in $\mathbb{R}^n$ with center at the origin. For every $k \in \{0, 1, \dots \}$, $\alpha \in (0, 1)$, and $q \in \mathbb{R}$, we define the weighted Hölder space $\mathcal{C}^{k,\alpha}_{-q}(\mathbb{R}^n \setminus B)$ as the collection of those $f \in \mathcal{C}^{k,\alpha}_{\text{loc}}(\mathbb{R}^n \setminus B)$ such that
\[
\|f\|_{\mathcal{C}^{k,\alpha}_{-q}(\mathbb{R}^n \setminus B)} := \sum_{|I| \leq k} \sup_{x} \left( |x|^{|I|+q} |\partial^I f(x)| \right) + \sum_{|I| = k} \left[ |x|^{\alpha+|I|+q} \partial^I f(x) \right]_\alpha < \infty,
\]
where $[\cdot]_\alpha$ denotes the Hölder seminorm and $\mathcal{C}^{k,\alpha}_{\text{loc}} $ the set of locally Hölder continuous functions.
\end{definition}
Suppose now that $M$ is a $\mathcal{C}^k$ manifold such that there is a compact set $K \subset M$ and a diffeomorphism $M \setminus K \cong \mathbb{R}^n \setminus B$. The space $\mathcal{C}^{k,\alpha}_{-q}(M)$ can then be defined just as we did for $W^{k,p}_{-q}(M)$ in the preceding definition.
\begin{definition}\label{AFtype}
     Let $ l\geq 3$ be an integer, we say that an initial data set $(M, g, K)$ is \emph{asymptotically flat} of type $(l,p, q, q_0, \alpha)$ if $g \in \mathcal{C}^{l,\alpha}_{\text{loc}}(M)$, $k \in \mathcal{C}^{l-1,\alpha}_{\text{loc}}(M)$, and if there is a compact set $K \subset M$ and a $\mathcal{C}^{l+1,\alpha}$  diffeomorphism $M \setminus K \cong \mathbb{R}^n \setminus B$ for some closed. ball $B \subset \mathbb{R}^n$ such that
\[
(g - \delta, k) \in W^{l,p}_{-q}(M) \times W^{l-1,p}_{-1-q}(M)
\]
where $\delta$ is a smooth symmetric $(0,2)$-tensor that coincides with the Euclidean inner product on $M \setminus K \cong \mathbb{R}^n \setminus B$, and such that
\[
\mu, J \in \mathcal{C}^{0,\alpha}_{-n-q_0}(M).
\]
\end{definition}
Corvino and Schoen showed in  \cite[Theorem 1]{CorvinoSchoen} that initial data sets with harmonic asymptotics are dense in the set of solutions of the set of asymptotically flat  vacuum  solutions of the constrained equations. This result was later generalized to solutions of the Einstein constraint equations under the dominant energy condition by  Eichmair,  Huang and Lee in \cite[Theorem 18]{eichmair2015spacetime}, this result can be stated assuming higher regularity like in Definition \ref{AFtype} (the proof of the result is basically unchanged).
\begin{theorem}\label{densitytheo}
    Let $(M, g, \pi)$ be an $n$-dimensional asymptotically flat initial data set of type $(l, p, q, q_0, \alpha)$ with $n\geq3$, $p>n$, $q_0>0$,  $\alpha \in (0 ,1-n/p]  $ and $q \in ((n-2)/2,n-2)$  such that the dominant energy condition $\mu \geq |J|_g$ holds. Let $\epsilon > 0$. There are asymptotically flat initial data $(\bar{g}, \bar{\pi})$ with harmonic asymptotics and of type $(l,p, q, q_0, \alpha)$ on $M$ such that
\[
\|g - \bar{g}\|_{W^{l,p}_{-q}} < \epsilon, \quad \|\pi - \bar{\pi}\|_{W^{l-1,p}_{-1-q}} < \epsilon, \quad |E - \bar{E}| < \epsilon, \quad |p_{ADM} - \bar{p}_{ADM}| < \epsilon,
\]
 where $E$ and $p_{ADM}$ denote the ADM energy and momentum,  and such that the strict dominant energy condition holds: $\bar{\mu} > |\bar{J}|_{\bar{g}}$.
\end{theorem}
Note that in particular, we have also a density result in the Hölder norm by the following result which can be found in \cite[Theorem A.25]{lee2021geometric}:
\begin{theorem}\label{sobolevemb}(Weighted Sobolev embeddings)
    Let $(M, g)$ be a $n$-dimensional  asymptotically flat manifold. Let $j$ be a positive integer, $p, q \geq 1$, and $s \in \mathbb{R}$. If $j - \frac{n}{p} < 0$ and $q \leq \frac{np}{n - jp}$, then $W^{j,p}_s(M) \subset L^q_s(M)$, and there exists a constant $C$ such that for all $u \in W^{j,p}_s$,
\[
\|u\|_{L^q_s} \leq C \|u\|_{W^{j,p}_s}.
\]
If $j - \frac{n}{p} \geq \alpha$ and $\alpha \in (0, 1)$, then $W^{j,p}_s(M) \subset \mathcal{C}^{0,\alpha}_s(M)$, and there exists a constant $C$ such that for all $u \in W^{j,p}_s$,
\[
\|u\|_{\mathcal{C}^{0,\alpha}_s} \leq C \|u\|_{W^{j,p}_s}.
\]
\end{theorem}

\vspace{1.5 cm}

 \paragraph*{\emph{Acknowledgements.}} The Author would like to thank Jan Metzger, Marc Mars, Rodrigo Avalos and Thomas Koerber for the helpful comments and discussions about this work. The author also thanks the anonymous referee for  helpful and constructive comments that improved this paper. This research received partial support from the DFG as part of the SPP 2026 Geometry at infinity, project ME 3816/3-1.

\bibliographystyle{amsplain}
\bibliography{Lit_new}
\end{document}